\documentclass[11pt]{amsart}
\usepackage{amsthm,amsmath,amsxtra,amscd,amssymb,xypic}
\usepackage[final]{graphicx}
\usepackage[all]{xy}

\makeatletter
\let\@wraptoccontribs\wraptoccontribs
\makeatother

\newcommand{\ssigma}{{\boldsymbol{\Sigma}}}
\newcommand{\Exp}{\boldsymbol{\operatorname{Exp}}}

\newcommand{\pr}{\operatorname{Pr}}
\newcommand{\stab}{{\operatorname{stable}}}
\newcommand{\rc}{{\operatorname{rc}}}
\newcommand{\Span}{\operatorname{Span}}
\newcommand{\stratum}[2][{g}]{\beta_{#1}({#2})}

\newcommand{\topd}{\operatorname{top}}

\newcommand{\psigma}[1]{P_{\sigma_{#1}}(t)}

\newcommand{\CC}{{\mathbb{C}}}

\newcommand{\HH}{{\mathbb{H}}}

\newcommand{\EE}{{\mathbb{E}}}

\newcommand{\QQ}{{\mathbb{Q}}}
\newcommand{\RR}{{\mathbb{R}}}
\newcommand{\TT}{{\mathbb{T}}}
\newcommand{\ZZ}{{\mathbb{Z}}}

\newcommand{\NN}{{\mathbb{N}}}

\newcommand{\calT}{{\mathcal T}}

\newcommand{\calA}{{\mathcal A}}

\newcommand{\calX}{{\mathcal X}}

\newcommand{\op}{\operatorname}
\newcommand{\ab}[1][g]{\calA_{#1}}
\newcommand{\ua}[1][g]{\calX_{#1}}
\newcommand{\Sat}[1][g]{{\calA_{#1}^{\op {Sat}}}}
\newcommand{\Vor}[1][g]{{\calA_{#1}^{\op {Vor}}}}
\newcommand{\Perf}[1][g]{{\calA_{#1}^{\op {Perf}}}}
\newcommand{\Centr}[1][g]{{\calA_{#1}^{\op {Centr}}}}
\newcommand{\Matr}[1][g]{{\calA_{#1}^{\op {Matr}}}}
\newcommand{\Std}[1][g]{{\calA_{#1}^{\op {Std}}}}

\newcommand{\Asigma}[1][g]{{{\calA}_{#1}^\ssigma}}
\newcommand{\Sp}{\op{Sp}}

\newcommand{\GL}{\op{GL}}

\newcommand{\Sym}{\op{Sym}}

\newcommand{\Aut}{\op{Aut}}

\newcommand\codim{{\rm codim}}
\newcommand\rank{\operatorname{rank}}

\newcommand{\pu}{\bullet}
\newcommand{\cohloc}[3]{H^{#1}(#2,#3)}
\newcommand{\coh}[2][\pu]{\cohloc {#1}{#2}{\QQ}}
\newcommand{\cohstab}[2][\pu]{H^{#1}_{\stab}({#2},\QQ)}

\newcommand{\Ht}[1][\pu]{H^{{\rm top}-#1}}

\newcommand{\BMloc}[3]{\bar H_{#1}(#2,#3)}
\newcommand{\BM}[2][\pu]{\BMloc {#1}{#2}{\QQ}}
\newcommand{\BMt}[2][\pu]{\BMloc {\topd-{#1}}{#2}{\QQ}}

\theoremstyle{plain}
\newtheorem{thm}{Theorem} %[section]
\newtheorem{lm}[thm]{Lemma}
\newtheorem{prop}[thm]{Proposition}
\newtheorem{cor}[thm]{Corollary}

\theoremstyle{definition}
\newtheorem{df}[thm]{Definition}
\newtheorem{rem}[thm]{Remark}

\title[Stable Betti numbers of toroidal compactifications]{Stable Betti numbers of (partial) toroidal compactifications of the moduli space of abelian varieties}
\author{Samuel Grushevsky}
\address{Mathematics Department, Stony Brook University,
Stony Brook, NY 11790-3651, USA}
\email{sam@math.stonybrook.edu}
\thanks{Research of the first author is supported in part by National Science Foundation under the grant DMS-15-01265.}
\author{Klaus Hulek}
\address{Institut f\"ur Algebraische Geometrie, Leibniz Universit\"at Hannover, Welfengarten 1, 30060 Hannover, Germany}
\email{hulek@math.uni-hannover.de}
\author{Orsola Tommasi}
\address{Mathematical Sciences, Chalmers University of Technology and the University of Gothenburg, SE-412 96 G\"oteborg, Sweden}
\email{orsola@chalmers.se}

\contrib[with an appendix by]{Mathieu Dutour Sikiri\'c}
\address{Rudjer Boskovi\'c Institute, Bijenicka 54, 10000 Zagreb, Croatia}
\email{mathieu.dutour@gmail.com}

\dedicatory{Dedicated to Nigel Hitchin on the occasion of his 70th birthday}

\begin{document}
\begin{abstract}
We present an algorithm for explicitly computing the number of generators of the stable cohomology algebra
of any rationally smooth
partial toroidal compactification of $\ab$ satisfying certain additivity and finiteness properties, in terms of the combinatorics of the corresponding toric fans. In particular the algorithm determines the stable cohomology of the matroidal partial compactification $\Matr$, in terms of simple regular matroids that are irreducible with respect to the 1-sum operation, and their automorphism groups.
The algorithm also applies to compute the stable  Betti numbers in close to top degree for the perfect cone toroidal compactification $\Perf$. This suggests the existence of an algebra structure on $\cohstab[\topd-k]{\Perf}$.
\end{abstract}

\maketitle

\section*{Introduction}

We work over $\CC$, and denote by $\ab$ the moduli space of complex principally polarized abelian varieties of dimension $g$. The classical result of Borel \cite{borel1} states that the cohomology of $\ab$ stabilizes: this is to say, $H^k(\ab,\QQ)$ is independent of $g$ for $g>k$. Moreover, in this range the cohomology is freely generated by the odd degree Hodge classes
$\lambda_{2k+1}:=c_{2k+1}(\EE)$, where $\EE$ denotes the Hodge bundle --- the complex rank $g$ vector bundle over $\ab$ whose fiber over $A$ is $H^{1,0}(A,\CC)$. We think of this result as computing the cohomology of $\ab[\infty]$ (in the sense of stability with respect to a sequence of stabilization maps, as discussed in Section~\ref{sec:stable}), and denote
$$
  R:=H^\pu(\ab[\infty],\QQ)=\QQ[\lambda_1,\lambda_3,\dots]
$$
this free polynomial algebra.

The moduli space $\ab$ admits various compactifications. Charney and Lee \cite{chle} proved that the cohomology of the Satake--Baily--Borel compactification $\Sat$ also stabilizes, i.e. that~$H^k(\Sat,\QQ)$ is independent of $g$ for $g>k$.
They proved, in particular, that the classes $\lambda_i$ can be extended (non-canonically) to $\Sat$.
Charney and Lee  proved that, as an $R$-algebra, the stable cohomology $H^\pu(\Sat[\infty],\QQ)$ is generated by certain classes $\alpha_3,\alpha_5,\dots$. Chen and Looijenga proved in \cite{chlo} that these classes have Hodge weight zero, and thus are non-algebraic. Thus the stable cohomology of $\Sat$ contains non-trivial Tate extensions, and we refer the reader to the recent preprint \cite{lo-pardon} for a discussion of these extensions.

In~\cite{grhuto} we investigated the stability of cohomology of toroidal compactifications or partial toroidal compactifications of $\ab$. The methods we used were different from the topological methods used by Borel, and Charney and Lee, and the results we obtained in~\cite{grhuto} were on stabilization in close to top degree.
It is easy to see that $\Ht[k](\Sat,\QQ)$ is independent of $g$ for $g>k$ (where here and below, $\operatorname{top}$ denotes the real dimension of the space, so in this case $g(g+1)$), and is freely generated by duals of the extensions of the odd Hodge classes. In \cite{grhuto}, we showed that cohomology of the perfect cone toroidal compactification $\Ht[k](\Perf,\QQ)$ is also independent of $g$ for $g>k$ --- we will thus say that its stabilizes in {\em codegree} up to $g$.

A primary goal of the current article is to provide an algorithm for computing such stable cohomology, in the more general setup of suitable partial toroidal compactifications.
As already observed in \cite{grhuto}, in the context of stabilization in small codegree, it is more natural to work with homology rather than with cohomology.
For spaces that are not rationally smooth, we typically work with Borel--Moore homology, i.e. homology with closed support (see \cite[Ch.~19]{fultonintersection}).
Specifically, our main result  is an algorithm that applies to Borel--Moore homology (or, dually, to cohomology with compact support) $\bar H_{\topd-k}(\Asigma)$ in codegree $k<g$ for any (admissible) small additive collection $\Asigma$ of partial toroidal compactifications of $\ab$.
These terms will be defined in detail in Section~\ref{sec:stable}, but essentially
{\em admissible} is the condition ensuring the existence of stabilization maps, while
{\em small} means that the dimension of each cone in the fan $\Sigma_g$, defining the partial toroidal compactification, is at least $\frac r2+1$, where $r$ is the {\em rank} of the cone (see Definition~\ref{def:rank} for the precise definition of rank).
Geometrically, this means that, if the stratum $\stratum\sigma\subset\Asigma$ corresponding to a cone $\sigma\in\Sigma_g$ maps to $\ab[g-k]\subset\Sat$, then its codimension in $\Asigma$ is at least equal to $\frac k2+1$.

{\em Additivity} is a property ensuring that all product maps $\ab[g_1]\times\ab[g_2]\to\ab[g_1+g_2]$
extend to product maps
\begin{equation}\label{eq:prod}
  \pr^\ssigma:\Asigma[g_1]\times\Asigma[g_2]\to\Asigma[g_1+g_2].
\end{equation}
Although this is a stronger property than admissibility, all known admissible collections of toroidal compactifications $\Asigma$ are also additive.

To be able to speak about the stable cohomology in a meaningful way, one needs to have the stabilization map relating cohomology for $g$ and $g+1$. The standard choice is to consider a stabilization map extending the map $\ab\rightarrow\ab[g+1]$ defined by taking products with a fixed elliptic curve. Provided the cone decompositions $\Sigma_g$ and $\Sigma_{g+1}$ satisfy some natural compatibility conditions, such an extension exists and defines the desired stabilization maps for the cohomology in a fixed degree.
If the stabilization map induces an isomorphism in degree $k$ for $g\gg k$, we say that the cohomology of $\Asigma$ stabilizes and we call the limit object the stable cohomology of $\Asigma$. If the compactification $\Asigma$ is additive, stable cohomology will have a Hopf algebra structure, where product and coproduct are induced by the usual cup product on $H^\pu(\Asigma)$ and the pull-back of $\Pr^\ssigma$ in the stable range, respectively. In this situation stable cohomology groups are always finite-dimensional, hence by Hopf's theorem stable cohomology must be isomorphic to a free graded-commutative algebra.

In our approach, we want to consider maps between cohomology (or homology) groups of fixed \emph{codegree}. Thus, we would like to work with maps that are Poincar\'e dual to the stabilization maps. Doing this presents an extra challenge, as the partial toroidal compactifications $\Asigma$ we consider are not necessarily rationally smooth. We circumvent this problem by proving in Section~\ref{sec:transverse} that the product maps $\pr^\ssigma$ are transverse embeddings (after passing to a level cover). This enables us to define Gysin maps in Borel--Moore homology of a fixed codegree.

The duals of the stabilization maps in Borel--Moore homology give the maps in cohomology with compact support of fixed codegree that were used in \cite{grhuto} for the perfect cone compactification. The reason why we prefer to work with Borel--Moore homology --- which we denote by $\bar H_\pu$ --- is that this is where the cycle maps naturally takes values. Namely, by \cite[Ch.~19]{fultonintersection}, for each complex scheme $X$ there is a well-defined cycle map $A_\pu(X)\rightarrow \bar H_\pu(X)$; we will call a Borel--Moore homology class {\em algebraic} if it lies in the image of the cycle map.
A further advantage of Borel--Moore homology is that the cap product gives $\BMt{\Asigma}$ an $\coh{\Asigma}$-module structure, and hence an $R$-module structure in the stable range.

In this language, the main result of~\cite{grhuto} combined with the transversality statement in Proposition~\ref{prop:transversal} gives the stabilization of Borel--Moore homology as follows.
\begin{thm}\label{thm:stable}
If $\ssigma=\lbrace\Sigma_g\rbrace_{g\ge 0}$ is a small admissible collection of fans, then there is an isomorphism
$$
  \BM[\topd-k]{\Asigma[\infty]}\cong  \BM[g(g+1)-k]{\Asigma}
$$
of Borel--Moore homology for all $k<g$. Moreover, in  this range this homology consists only of algebraic classes.
\end{thm}

For additive collections, the existence of transverse product maps ensures that stable homology carries a coalgebra structure -- and, if $\Asigma $ is rationally smooth, a Hopf algebra structure isomorphic to that of a polynomial algebra with generators of even degree.
Our goal is to determine the number of generators of the algebra in each degree, and our main result is the following:
\begin{thm}\label{thm:algorithm}
If $\ssigma=\lbrace\Sigma_g\rbrace_{g\ge 0}$ is a small additive collection of fans,  then the stable Borel--Moore homology $\BM[\topd-\pu]{\Asigma[\infty]}$ is isomorphic to the free $R$-module generated by the symmetric algebra $\Sym^\pu(V^\ssigma)$ of the graded $\QQ$-vector space $V^\ssigma$ which is trivial in odd degree and given in even degree by
$$
  V^\ssigma_{2k}:=\bigoplus_{\substack{[\sigma]\in[\ssigma]\\\sigma \text{ irreducible }}}
(\Sym^{k-\dim\sigma} V_\sigma)^{G_\sigma},
 $$
where
\begin{itemize}
\item $[\ssigma]:=\underset{g}{\underrightarrow{\lim}}\;\Sigma_g/\GL(g,\ZZ)$ is the set of all orbits of cones in $\ssigma$;
\item $\sigma$ is called irreducible if it is not equal to a direct sum of two  non-zero cones of $\ssigma$;
\item $V_\sigma$ denotes the $\QQ$-span of $\sigma$; $G_\sigma$ denotes the refined automorphism group of $\sigma$ (see Definition~\ref{d:span,auto}), and we are thus summing the invariant subspaces of $\Sym^\pu V_\sigma$ under the action of $G_\sigma$.
\end{itemize}
\end{thm}
We note that the theorem only gives the description of the stable Borel--Moore homology as an $R$-module; it does not provide a canonical choice of generators as geometrically identified classes on the partial toroidal compactifications $\Asigma$. However, the description given is completely explicit, and yields an algorithm to compute the dimensions of the stable homology groups.

One particular case of interest is the perfect cone toroidal compactification $\Perf$, which was the main focus of~\cite{grhuto}. We note that a priori $\BM[\topd-k]{\Perf[\infty]}$ does not carry any algebra structure, because the collection of homology groups  $\BM[g(g+1)-\pu]{\Perf}$ is not a ring. However, our result shows that in the stable range $k<g$ the homology $\BM[g(g+1)-k]{\Perf}$ has the same Betti numbers as a polynomial algebra. This suggests that stable homology and cohomology in close to top degree admit an algebra structure, eg.~via a lift to intersection homology.

We stated Theorem~\ref{thm:algorithm} as a description of stable homology as an $R$-module, and thus need to explain how the Hodge classes extend to various partial compactifications of $\ab$. As discussed above, there is the issue of choosing a suitable extension of the Hodge classes to $\Sat$; to deal with this, one can choose
a compatible collection of extensions of the classes  $\lambda_i$ to $\Sat$ for all genera $g$ (where by compatible we mean that for any $0<k<g$ and for any $i$ the extension of $\lambda_i$ defined on $\Sat$ pulls back to the extension of $\lambda_i$ defined on $\Sat[k]$, under the map taking a product with a fixed ppav of dimension $g-k$). For (partial) toroidal compactification, there is a more direct way to extend the Hodge classes. Indeed, the Hodge rank $g$ vector bundle on $\ab$ extends as a vector bundle to any partial toroidal compactification $\ab^{\ssigma}$ --- this is proven in full generality in~\cite[Thm V.2.3]{fachbook}, while the analytic argument for rationally smooth toroidal compactification is given in~\cite{mumhirz}. Thus on any (partial) toroidal compactification $\ab^{\ssigma}$ the extensions of the Hodge classes $\lambda_i$ can be defined as the Chern classes of the extension of the Hodge vector bundle.
Thus all stable cohomology groups we consider in this paper will be $R$-algebras, and all stable homology groups (in small codegree) will be $R$-modules, and we formulate our results in these terms.

\smallskip
While our main theorem above holds for any small additive collection, the situation is particularly good if  the (partial) compactification $\Asigma[\infty]$ is rationally smooth. This is the case if and only if all cones $\sigma \in \Sigma$ are simplicial, see eg.~\cite[Theorem 11.4.8]{coxlittleschenck}. Then, by Poincar\'e duality, the cap product with the fundamental class of $\Asigma$ defines an isomorphism between cohomology in degree $k$ and Borel--Moore homology in codegree $k$.

One such case of particular interest are  the matroidal partial toroidal compactifications $\Matr$. These were introduced and first studied by Melo and Viviani~\cite{mevi} who showed that the cones defined starting from simple regular matroids, and forming the fan $\ssigma^{\op{Matr}}$, are in fact the intersection $\Sigma_g^{\op{Matr}}=\Sigma_g^{\op{Perf}}\cap\Sigma_g^{\op{Vor}}$  of the perfect cone and second Voronoi fans. These partial compactifications $\Matr$ are rationally smooth (all cones of $\ssigma^{\op{Matr}}$ are simplicial by~\cite[Theorem 4.1]{erry3}). In this case the theorem above takes an especially explicit form, as the irreducibility of the cone is governed by the properties of the corresponding matroid. Before we formulate this, we recall that the rank
of a matroid $M$ is the number of elements in any of its bases, and that this coincides with the dimension of the corresponding matroidal cone $\sigma_M$. We also recall that the $1$-sum of matroids corresponds to taking direct sum of the corresponding cones.
\begin{cor}\label{cor:Matr}
The stable cohomology $H^\pu(\Matr[\infty])$ is the free $R$-algebra generated by  $\Sym^\pu(V^{\op{Matr}})$. Here $V^{\op{Matr}}$ is the graded vector space which  is trivial in odd degree and is given in even degree by
$$
V^{\op{Matr}}_{2k}=\bigoplus_{\substack {\text{$M\in ${\{isomorphism classes of simple}}\\ \text{1-sum-irreducible regular matroids\}}}}
(\Sym^{k-\rank M}(\Span_\QQ(\sigma_M))^{\Aut(M)}
$$
where $\sigma_M$ denotes the cone in the matroidal fan corresponding to $M$, and $\Aut(M)$ denotes the automorphism group of the matroid $M$.
\end{cor}
This corollary provides an explicit combinatorial algorithm to compute the stable cohomology of the matroidal partial compactification in any degree, once 1-sum-irreducible simple regular matroids in up to that rank are classified.

While we do not have an analogous combinatorial closed formula for the perfect cone toroidal compactification $\Perf$, our algorithm can still be implemented in that case.
In Section~\ref{sec:examples} we compute the stable cohomology of $\Matr[\infty]$ in degree up to 16, and stable Borel-Moore homology of $\Perf[\infty]$ in codegree up to 16 (note that $\Perf[\infty]$ is not rationally smooth), by using the known enumerations of the corresponding cones, and studying their automorphism groups and invariants by hand. In the appendix Dutour Sikiri\'c explains details about the enumeration of cones in the perfect cone decomposition, using the results of~\cite{onlinelistcones}, \cite{DutourHulekSchuermann} and \cite{numberofperfectforms} on classification of perfect cones.  He then gives a brief overview of his new computational methods and implementation of our algorithm, which allow to push these computations considerably further. The final results of his computation are the following two theorems.
\begin{thm}\label{thm:bettimatr}
The stable Betti numbers of $\Matr$ in degree $\leq 30$ are as follows:
$$
\begin{array}{l|rrrrrrrrrr}
\text{\rm degree $k$}&0&2&4&6&8&10&12&14&16&18\\
\hline
\\[-2ex]
\dim H^k(\Matr[\infty])&
1&2&4&9&18&37&79&169&379&902
\end{array}
$$

$$
\begin{array}{l|rrrrrr}
\text{\rm degree $k$}&20&22&24&26&28&30\\
\hline
\\[-2ex]
\dim H^k(\Matr[\infty])&
2287&6270&18864&62466&228565& 920313
\end{array}.
$$
\end{thm}

\begin{thm}\label{thm:bettiperf}
The stable Betti numbers of $\Perf$ in codegree $\leq 22$ are as follows:
$$
\begin{array}{l|rrrrrrrrrrr}
 \text{\rm codegree } k&2&4&6&8&10&12&14&16&18&20&22\\
 \hline
 \\[-2ex]
 %\dim\BM[\topd-k]{\Perf[\infty]}
\bar H_{\topd-k}(\Perf[\infty])
 & 2& 4& 9& 18& 38& 84& 193& 494&1529&6791&70464
\end{array},
$$
where the result in degree $22$ is conjectural (but certainly gives a lower bound for this dimension).
\end{thm}
For both of these theorems, the results in (co)degree up to 12 were computed in~\cite[Thm 1.6]{grhuto} with much more work, and match what we obtain now, except for the discrepancy in degree 12. This mismatch is due to an unfortunate error in~\cite[Table 1]{grhuto}, where the correctly obtained result of \cite[Corollary 10.3]{grhuto} is reported incorrectly in the table. The theorems above thus extend the computations done in~\cite{grhuto}, in a much quicker way, up to degree 30 for $\Matr[\infty]$, and up to codegree $22$ for $\Perf[\infty]$.

The reason that  $\dim\BM[\topd-22]{\Perf[\infty]}$ is only known conjecturally is due to the fact that the classification of configurations of vectors of dimension 11 and rank 11 is not finished, see the appendix for a discussion.

The computations above in fact yield more information --- they allow us to compute the stable Borel-Moore homology of subsets of $\Matr[\infty]$ and  $\Perf[\infty]$ that correspond to cones that are direct sum of cones of rank up to 15 and respectively 11. From the degeneracy in the stable range of the suitable Gysin spectral sequence (for Borel-Moore homology), such as used in~\cite{grhuto}, it follows that stable Borel-Moore homology of any open union of toroidal strata
of a partial toroidal compactification is a subspace of the stable Borel-Moore homology of the original partial compactification. Thus in fact our methods gives explicit lower bounds for the stable cohomology of $\Matr$ in any degree, and for the stable Borel-Moore homology of $\Perf$ in any codegree. The generating functions for these lower bounds are available on \cite{onlinelistcones}.

\subsection*{The structure of the paper}
In Section \ref{sec:stable} we recall the definition and properties of an admissible collection of admissible fans. We then introduce two properties of such collections which are the basis of our work, namely {\em additivity} and {\em smallness}. The first of these properties guarantees that
the toroidal compactifications behave well with respect to the product embedding $\ab[k] \times \ab[g-k] \to \ab$. This will be crucial for our computational approach. The latter condition is what is required to guarantee stabilization of the Borel-Moore homology in the first place. We also rephrase the results of \cite{grhuto} in the form needed for the development of our algorithm.

In Section \ref{sec:compute} we develop the techniques for working with representations of wreath products. This is the essential technical tool which will allow us to make use of the additive structure of fans. Next, we recall the notion of  plethystic substitution in Section \ref{sec:proof}. This will enable us to state our main result, Theorem \ref{thm:algorithm}, in a compact form, and give its proof. The transversality statement for product embeddings will be proven in Section~\ref{sec:transverse}. Finally, we will discuss some examples in Section \ref{sec:examples} in order to demonstrate the effectiveness of our algorithm, starting from the additive collection of ``standard'' cones (i.e.~those of the form $\sigma_{1+\dots+1}$), which now reduce to a one-line computation, and then proceeding to compute the stable Betti numbers of $\Perf$ for codegree up to 16. In the appendix Dutour Sikiri\'c explains his method and implementation of an algorithm that allows him to push these computations to degree up to 30 for $\Matr[\infty]$ and to codegree up to 22 for $\Perf[\infty]$.

\section{Stable partial compactifications and product maps}\label{sec:stable}
In this section we recall the data needed to define a partial toroidal compactification of $\ab$, and introduce various properties of such compactifications required for our results to apply. We refer to~\cite{amrtbook} and~\cite{namikawabook} for the basic theory.

The construction of a toroidal compactification consists of two steps: firstly, one has to construct a compactification for each cusp, and secondly one has to glue these partial compactifications in order to obtain a global compactification $\Asigma$. In the case of principal polarizations, and this is the situation which we will always be in, there is, up to the action of $\Sp(2g,\ZZ)$,  only one cusp for each integer $k < g$,
a fact which we will use frequently. These cusps are in $1$-to-$1$ correspondence with the (orbits of) isotropic subspaces $\QQ^k$ of $\QQ^{2g}$. As a model we can take the
subspace generated by the first $k$ standard basis vectors.

In order to describe the data required to construct a partial compactification  in the direction of such a cusp, we recall that
a cone in a real vector space is called rational polyhedral if it is generated, over $\RR$, by finitely many rational vectors. We denote by $\Sym^2_{>0}(\RR^k)$ the cone of positive-definite symmetric $k\times k$ real matrices, and denote by $\Sym^2_{\rc}\RR^k$ its rational closure: the cone of positive semidefinite symmetric $k\times k$ real matrices such that the kernel has a basis consisting of vectors in $\QQ^k$.  Neither of these two cones is rational polyhedral for $k\ge 2$, and the fan is used to describe a cover of $\Sym^2_{>0}(\RR^k)$ by rational polyhedral cones contained in $\Sym^2_{\rc}\RR^k$.
\begin{df}\label{def:admfan}
An {\em admissible fan}
$\Sigma_k$ is a collection of rational polyhedral cones $\sigma\subset\Sym^2_{\rc}(\RR^k)$ such that
\begin{itemize}
\item[(i)] If $\sigma\in\Sigma_k$ is a cone, then any face $\tau$ of $\sigma$ is also a cone in $\Sigma_k$.
\item[(ii)] If $\sigma,\tau\in\Sigma_k$, then also the intersection $\sigma\cap\tau$ is a cone in $\Sigma_k$.
\item[(iii)] The union of all cones $\sigma\in\Sigma_k$ contains $\Sym^2_{>0}(\RR^k)$.
\item[(iv)] The collection of cones $\Sigma_k$ is preserved under the action of $\GL(k,\ZZ)$ on $\Sym^2_\rc(\RR^k)$, and the number of orbits of cones in $\Sigma_k$ under this action is finite.
\end{itemize}

Moreover, we say that the admissible fan $\Sigma_k$ is {\em complete} if it satisfies
\begin{itemize}
\item[(v)]
 The union of cones in $\Sigma_k$ is equal to $\Sym^2_\rc(\RR^k).≈$
  \end{itemize}
\end{df}

Given such a  fan, we can construct a partial compactification in the direction of the cusp corresponding to $\QQ^k$.
As we have said, the different partial compactifications need to be compatible with each other. This leads to the notion of an admissible
collection of admissible fans.
As we will work with arbitrarily large $g$, we will give the relevant definition directly for an infinite sequence of admissible fans.

\begin{df}\label{def:admcollection}
A sequence $\ssigma=\lbrace\Sigma_k\rbrace_{k\ge 0}$ is called an {\em admissible collection} if
\begin{itemize}
\item[(i)] for any $k\ge 0$, $\Sigma_k$ is an admissible fan in $\Sym^2_{\rc}(\RR^k)$.
\item[(ii)] for any $k<k'$ the equality
  $$
     \Sigma_k=\lbrace\sigma\in\Sigma_{k'} \mid\sigma\subset\Sym^2_{\rc}\RR^k\rbrace
  $$
  holds for one (and hence any) coordinate embedding $\RR^k\hookrightarrow\RR^{k'}$ induced by a primitive embedding $\ZZ^k \hookrightarrow\ZZ^{k'}$.
\end{itemize}
\end{df}
Choosing an admissible collection of admissible fans thus allows us to construct (partial) compactifications $\Asigma$ for each $g$. A given compactification $\Asigma$ is compact if and only if fan is complete, as in Definition \ref{def:admfan}.(v). Every toroidal compactification $\Asigma$ allows a contraction morphism  $\phi_\ssigma: \Asigma \to \Sat$. Since $\Sat = \ab[g] \sqcup \ab[g-1] \sqcup \dots \sqcup \ab[0]$, we can consider the preimages  $\phi_\ssigma^{-1}(\ab[g-k])$. The fan $\Sigma_k$ gives a stratification of each such preimage into strata corresponding to the (orbits of) cones in $\Sigma$. As we shall see below, these strata can be described explicitly in terms of the cones $\sigma$.
%These conditions guarantee that for all $g$, the partial compactifications $\phi^{-1}(\ab[g-k])$ defined using $\Sigma_k$ match to define a partial toroidal compactification $\Asigma$ of $\ab$. A given admissible collection $\ssigma=\lbrace\Sigma_g\rbrace_{g\ge 0}$ thus defines a partial toroidal compactification $\Asigma$ for any $g$. If $\Sigma_g$ satisfies (v), then $\Asigma$ is compact.
Finally, we  remark that admissibility implies the existence of natural maps $\Asigma \rightarrow \Asigma[g+1]$ for all $g$, given by taking a product with some fixed elliptic curve $E$.

Our goal is to study stability, and for this one needs stabilization maps. The simplest map to consider is the map $\ab[g_1]\times\ab[g_2]\to\ab[g_1+g_2]$ that sends a pair of principally polarized abelian varieties to their product. We want this map to extend to $\Asigma$, as in~\eqref{eq:prod}.

\begin{df}
An admissible collection $\ssigma=\lbrace \Sigma_g\rbrace_{g\ge 0}$ is called an {\em additive collection} if for any $g_1,g_2$, and for any $\sigma_1\in\Sigma_{g_1}$ and $\sigma_2\in\Sigma_{g_2}$, the direct sum of the cones $\sigma_1\oplus\sigma_2$ is a cone in $\Sigma_{g_1+g_2}$.
\end{df}
From the construction of partial toroidal compactifications it follows that for any additive collection, the product maps~\eqref{eq:prod} indeed extend, and are in fact transverse. We will prove the following result in the Section \ref{sec:transverse}.
\begin{prop}\label{prop:transversal}
Let $\lbrace\Sigma_g\rbrace$ be an additive collection of admissible fans. Then for any $0\leq k\leq g$ the product map $\ab[k] \times\ab[g-k]\to\ab$
extends, after going to a suitable level structure, to a transverse embedding $\Asigma[k] \times \Asigma[g-k] \to \Asigma$.
\end{prop}
\begin{rem}
We can take any  level cover such that the corresponding arithmetic group is neat, in particular we can take a full level-$n$ cover for $n \geq 3$.
\end{rem}
All three known admissible collections of complete fans, namely the perfect cones fan $\Sigma_g^{\operatorname{Perf}}$, second Voronoi fan $\Sigma_g^{\operatorname{Vor}}$, and central cones fan $\Sigma_g^{\operatorname{Centr}}$, are additive. Furthermore, the admissible collection of matroidal cones $\Sigma_g^{\operatorname{Matr}}:=\Sigma_g^{\operatorname{Perf}}\cap\Sigma_g^{\operatorname{Vor}}$ (see~\cite{mevi}) is also additive --- which is immediate from the fact that there is a well-defined 1-sum operation for matroids, and of course also follows from the fact that both $\Sigma_g^{\operatorname{Perf}}$ and $\Sigma_g^{\operatorname{Vor}}$ are additive.
We note that all three toroidal compactifications $\Perf,\Vor,\Centr$ fail to be  rationally smooth for $g \geq 4$, while $\Matr$ is rationally smooth for any $g$.
The latter follows from \cite[Theorem 4.1]{erry3}. For a discussion of the singular loci of these toroidal compactification we also refer the reader to \cite{DutourHulekSchuermann}.

The product map we are particularly interested in is taking the product of a ppav with a fixed elliptic curve $E\in\ab[1]$, which gives an embedding $\Asigma\to\Asigma[g+1]$ for any admissible collection of fans $\ssigma$.
While this map depends on $E$, its homotopy class is independent of the choice of $E$. Being a transverse embedding (which can easily be checked with the arguments given in Section \ref{sec:transverse}) it defines, in view of~\cite[Ex. 19.2.1]{fultonintersection}, a Gysin map on Borel--Moore homology
\begin{equation}\label{eq:stablemap}
 \BM[(g+1)(g+2)-\pu]{\Asigma[g+1]} \rightarrow \BM[g(g+1)-\pu]\Asigma.
\end{equation}

Thus given an admissible collection, we have a sequence of Gysin maps, of which we can then take the inverse limit, writing it as
$$
  \BM[\topd-\pu]{\Asigma[\infty]}:=  \underset{g}{\underleftarrow{\lim}}\; \BM[g(g+1)-\pu]\Asigma.
$$
We will call this  {\em stable Borel--Moore homology}. We recall that since the Hodge vector bundle extends as a vector bundle to any partial toroidal compactification, the ring $R$ naturally extends to cohomology classes on $\Asigma$, and the stable Borel--Moore homology is thus an $R$-module.

If moreover $\ssigma$ is also additive, then the Gysin maps associated to the product maps define a coproduct structure
$$
  \BM[\topd-\pu]{\Asigma[\infty]} \longrightarrow
  \BM[\topd-\pu]{\Asigma[\infty]}\otimes \BM[\topd-\pu]{\Asigma[\infty]},
$$
where the right hand side uses K\"unneth formula for $\Asigma[g_1]\times\Asigma[g-g_1]$ in the stable range. When $\Asigma$ is rationally smooth for all $g$, i.e. if $\ssigma$ is simplicial, this is the coproduct that, together with the cup product, gives stable cohomology the structure of a Hopf algebra.

For arbitrary admissible collections, the stable homology defined above may be infinite-dimensional. Indeed, for example it is known that the
inequality $\dim H^2(\Vor,\QQ)\ge g-3$ holds (see the discussion in~\cite[\S7]{grhuto}), and thus $H^2(\Vor[\infty],\QQ)$ is infinite-dimensional. The main result of~\cite{grhuto} is a proof that for $\ssigma^{\operatorname{Perf}}$ the stable homology is finite-dimensional, and that the maps~\eqref{eq:stablemap} are isomorphisms for $\pu<g$. By inspection of the proof, the argument in~\cite{grhuto} proves this for a more general class of partial toroidal compactifications given by additive collections.

\begin{df}\label{def:rank}
For an admissible collection $\ssigma$, we define the {\em rank} of a cone $\sigma\in\Sigma_g$ to be the minimal $k\ge 0$ such that there exists a cone $\tau\in\Sigma_k=\Sigma_g\cap\Sym^2(\RR^k)$ lying in the $\GL(g,\ZZ)$-orbit of $\sigma$.
\end{df}

As usual we will define the dimension of a cone $\sigma$ as the smallest dimension of a linear subspace containing $\sigma$.

\begin{df}\label{df:small}
An admissible collection $\ssigma$ is called \emph{small} if for any cone $\sigma\in\ssigma$ of rank $\geq 2$, the inequality $\dim\sigma\ge\frac{\rank\sigma}2+1$ holds.
\end{df}
One does not impose any condition on cones of rank $1$ because they are necessarily $1$-dimensional.

\begin{rem}\label{rem:small1}
Let $\ssigma$ be a small admissible collection.
We claim that for given $k\geq 0$ the number of $\GL(g,\ZZ)$-orbits of cones of dimension $k$ depends only on $k$, but not on $g$, provided $g$ is sufficiently large.
Let us assume $g\geq 2k-2$ holds.
Then, by the definition of a small collection, any cone of dimension $k$ has rank at most $2k-2$ and is thus
$\GL(g,\ZZ)$-equivalent to a cone in $\Sigma_{2k-2}$. Hence the number of $\GL(g,\ZZ)$-orbits of cones of dimension $k$ is the same as the number of
$\GL(2k-2,\ZZ)$-orbits of such cones, and this is clearly independent of $g$.
%. Hence the number of
%$\GL(g,\ZZ)$-orbits of cones of dimension $k$ equals the number of $\GL(k.\ZZ)$ orbits of cones of dimension $k$ in $

%Then for all  $g\geq k\geq 0$,
%the set of all cones of dimension $k$ in $\Sigma_g$ is contained in the set of all cones of rank at most $k$. By admissibility, all cones of rank $\leq k$ have a representative in $\Sigma_k$, which yields that the number of $\GL(g,\ZZ)$-orbits of cones of dimension $k$ in $\Sigma_g$ equals
%$$
%\#\Big(\{\sigma\in\Sigma_k|\; \dim\sigma=k\}/\GL(k,\ZZ)\Big)
%$$
%and is thus independent of $g\geq k$.
\end{rem}
\begin{rem}\label{rem:small2}
We note that for the only known small additive collection of (full, as opposed to partial) compactifications, that is for the perfect cone compactifications $\Perf$, the stronger property that codimension $k$ strata occur over $\ab[g-k]$ holds --- which then implies that all cones $\sigma$ satisfy $\dim\sigma\geq\rank\sigma$. The proof of stabilization of cohomology of $\Perf$ in close to top degree, given in~\cite{grhuto} used this stronger property in a crucial way; our current setup with small additive collections is more general, and it would be interesting to discover new examples where it may apply.
\end{rem}
The main result of~\cite{grhuto} can be made to work for  any small admissible collection, and in our current setup can be phrased as follows.
\begin{thm}\label{thm:oldmain}
For any small admissible collection $\ssigma$ the equality
$$
  \BM[\topd-k]{\Asigma[\infty]}\cong  \BM[g(g+1)-k]{\Asigma}
$$
holds for all $k<g$. Furthermore, in this range the Borel--Moore homology of $\Asigma$ consists entirely of algebraic classes.
\end{thm}
In fact the results in~\cite{grhuto} are stated in terms of cohomology with compact support, but (see also~\cite[Rem.~9.2]{grhuto}) they can be better rephrased in terms of Borel--Moore homology, so that the concept of algebraic class is well-defined.
\begin{proof}
The proof of Theorem~\ref{thm:oldmain} in~\cite{grhuto} proceeds by first showing that for a given cone $\sigma\in\Sigma_k$, the cohomology of the corresponding stratum $\beta_g(\sigma)\subset\Asigma$ stabilizes as $g$ increases (note that admissibility implies that $\sigma$ is also a cone in $\Sigma_g$ for any $g\ge k$), then showing that this cohomology is purely algebraic and in particular non-zero only in even degree. This implies that the Gysin spectral sequence for the stratification of $\Asigma$ by $\stratum\sigma$ degenerates at $E^1$ in the stable range, which finally shows that Borel--Moore homology in small codegree of $\Asigma$ is the sum of the Borel--Moore homologies in small codegree of $\stratum\sigma$, and thus stabilizes.

For $p\leq g$, let us denote by $\gamma_p$ the minimal dimension of a rank $p$ cone in $\ssigma$. (If no rank $p$ cone exists, we may set $\gamma_p=\infty$.) Then $\gamma_p$ is also equal to the codimension of the union of all strata $\stratum\sigma$ with $\rank\sigma = p$, that is of the locus
$$
  \beta^0_{p,g}:=\bigsqcup_{\substack{[\sigma]\in[\ssigma]\\\rank\sigma=p}}\stratum\sigma,
$$
which is a (possibly reducible) locally closed subset of $\Asigma$.

Recall from~\cite[Rem. 9.4]{grhuto} that the Borel--Moore homology of $\beta^0_{p,g}$ stabilizes in degree $g(g+1)-2\gamma_p-k$ for $k< g-p-1$. For $p=0,1$, however, the strata $\beta^0_{p,g}$ are isomorphic to $\ab$ and the universal Kummer variety $\calX_{g-1}/\pm1$, respectively, so that the stability range is in degree $g(g+1)-k$ for $k<g$ in case $p=0$ (by Borel's stability theorem) and in degree $g(g+1)-2-k$, for $k<g-1$ in case $p=1$ (by~\cite[Prop. 4.3]{grhuto}). In the stable range, the Borel--Moore homology of $\beta^0_{p,g}$ is always algebraic and in particular it vanishes in odd degree.

From now on, we assume that $\ssigma$ is small, so that we have $\gamma_p\leq \frac p2+1$ for all $p\geq 2$, or equivalently $2\gamma_p-p-2\geq 0$.  Let us consider the Gysin spectral sequence associated with the stratification $\{\beta^0_{p,g}\}$. We write it by using the rank $p$ of the stratum and the codegree $q$ as natural parameters, as follows:
$$
   E^1_{-p,-q}=\BM[g(g+1)-p-q]{\beta^0_{p,g}}\Rightarrow \BM[g(g+1)-p-q]{\Asigma}.
$$
In view of the considerations above, the term $E^1_{-p,-q}$ stabilizes in the following cases:
$$\begin{array}{l@{}l}
p=0& \text{ and } q<g,\\
p=1&  \text{ and } q<g,\\
p\geq 2& \text{ and }   p+q<2\gamma_p+g-p-1.
\end{array}            $$
Within this stable range one has $E^\infty_{-p,-q}=E^1_{-p,-q}=\BM[g(g+1)-p-q]{\beta^0_{p,g}}=0$ if $p+q$ is odd.

Next, we study the differential
$$d^r:\; E^r_{-p,-q}\longrightarrow E^r_{-p-r,-q+r-1}.$$
It is easy to see that for $p+q<g$ and all $r\geq 1$, both $E^1_{-p,-q}$ and $E^1_{-p-r,-q+r+1}$ lie in the stable range defined above.
 To prove this, it suffices to check that $E^1_{-p',-q'}$ lies in the stable range for $p'=p+r$ and $q'=q-r+1$. Indeed, in this case we have $p'+q'\leq g$. If $p'=1$ holds, this implies $q< g$ and we are automatically in the stable range. For $p'\geq 2$, smallness implies
$$
p'+q'\leq g \leq g+2\gamma_{p'}-p'-2<2\gamma_{p'}+g-p'-1.
$$

Hence, for $p+q<g$ one has $E^r_{-p,-q}=0$ if $p+q$ is odd and $E^r_{-p-r,-q+r-1}=0$ for $p+q$ even. In both cases, the differential $d^r$ vanishes.
This proves that the Gysin spectral sequence degenerates at $E^1$ in codegree $k=p+q<g$. This implies the claim.
\end{proof}

\begin{rem}\label{rem:small3}
In fact it is possible to weaken the smallness assumptions even further, if one is interested in stabilization in codegree $k$ for $g$ sufficiently large, but not necessarily for $k<g$. For instance, the proof above can be adapted to show that for a fixed integer $k_0$, the Borel--Moore homology of $\Asigma$ stabilizes in codegree $k< k_0$ whenever
$$
  \dim\sigma< \frac{k_0+\rank \sigma-g+1}2
$$
holds for all cones $\sigma\in\ssigma$. It would be interesting to see if there are natural admissible collections of partial compactifications that satisfy this weaker smallness assumption for some (large) integer $k_0$.
\end{rem}
As the goal of the current paper is to provide an algorithm for computing the stable homology, we now recall the description of the cohomology of the stratum $\beta(\sigma)\subset\Asigma$ corresponding to a cone $\sigma\in\Sigma_g$. Recall that the partial toroidal compactification decomposes as
$$
\Asigma = \bigsqcup_{[\sigma]\in \Sigma_g/\GL(g,\ZZ)}\stratum\sigma,
$$
where $\stratum\sigma$ is a locally closed subset of $\Asigma$ of codimension $\dim\sigma$, and the image of $\stratum\sigma$ under the contracting morphism $\phi_\ssigma$ to the Satake compactification is contained in $\ab[g-\rank\sigma]$. Thus the condition of smallness is equivalent to requiring $\codim\,\phi_\ssigma^{-1}(\ab[g-k])\ge 2k-2$ for all $0\leq k\leq g$. Each stratum $\stratum\sigma$ admits an explicit description as a torus bundle over a fiber product of the universal family over $\ab[g-k]$; we refer to the explicit discussion in~\cite[\S8]{grhuto}.
In particular, each $\stratum\sigma$ is rationally smooth, and thus satisfies Poincar\'e duality over $\QQ$. For any additive collection $\ssigma$, the stability maps are compatible with this stratification, i.e.~they act stratum-wise and restrict to maps $\stratum\sigma\rightarrow\stratum[g+1]\sigma$. The cohomology of $\stratum\sigma$ can then be computed in terms of invariants of the automorphism group of $\sigma$, as we now recall.

\begin{df}\label{d:span,auto}
For a cone $\sigma\in\Sigma_g$ in an admissible collection, denote $V_\sigma:=\Span_\QQ(\sigma)$ the $\QQ$-span of $\sigma$. The {\em refined automorphism group} of $\sigma$, denoted $G_\sigma$, is defined to be the image in $\GL(V_\sigma,\QQ)$ of the stabilizer of $\sigma$ in $\GL(g,\ZZ)$.
\end{df}
For a given cone $\sigma\in\Sigma_k$, the stable cohomology of $\stratum\sigma$ has been computed as follows.
\begin{thm}[{\cite[Th~8.1]{grhuto}}]\label{thm:priorresult}
\begin{itemize}
\item[(i)] Let $\sigma\in\Sigma_k$. Then for any genus $g\geq k$ and any degree $i<g-\rank \sigma -1$ the cohomology group $\coh[i]{\stratum[g]\sigma}$ consists only of algebraic classes, and the stability map $\coh[i]{\stratum[g+1]\sigma}\rightarrow \coh[i]{\stratum[g]\sigma}$ is an isomorphism.
\item[(ii)] The stable cohomology algebra $\coh{\stratum[\infty]\sigma}:=\underset{g}{\underleftarrow{\lim}}\;\coh{\stratum\sigma}$ is isomorphic to the free $R$-algebra generated by the $G_\sigma$-invariants of the symmetric algebra of $V_\sigma$, i.e.
$$
  \coh{\stratum[\infty]\sigma} \cong\left(\Sym^\pu V_\sigma\right)^{G_\sigma}\otimes_\QQ R,
$$
where we assign degree $2$ to the generators of $V_\sigma$.
\end{itemize}
\end{thm}
The computation of stable Borel-Moore homology of $\Perf$, and more generally for any small additive collection, is then completed by the observation of the degeneracy of the spectral sequence.
\begin{prop}[{\cite{grhuto}}]\label{prop:degenerate}
For codegree $k<g$ the spectral sequence in Borel--Moore homology associated to the stratification of $\Asigma$ into strata $\stratum\sigma$ degenerates at $E^1$. Thus in particular $$\BM[\topd-\pu]{\Asigma[\infty]}\cong\bigoplus_{[\sigma]\in[\ssigma]}\BM[g(g+1)-2\dim\sigma-\pu]{\stratum[\infty]\sigma}.$$
\end{prop}
\begin{proof}
The claim follows by combining Proposition 9.3 and Lemma 9.5 of \cite{grhuto}.
\end{proof}

The formula above already gives a possible approach to computing the stable Borel--Moore homology of $\Asigma$ in close to top degree. One can compute the stable cohomology $H^\pu(\stratum[\infty]\sigma,\QQ)$ for any cone $\sigma$ using invariant theory; then the stable Borel--Moore homology of $\Asigma$ in small codegree will be given by the sum of these for all $\sigma\in\Asigma$.

At the same time, this is deeply unsatisfactory if $\ssigma$ is additive, especially in the case in which $\Asigma$ is rationally smooth. In this case, by Poincar\'e duality, the stable Borel--Moore homology in small codegree is isomorphic to stable cohomology in small degree, which has a natural ring structure. Moreover, the product maps $\Pr^\ssigma$ induce a coalgebra structure on stable cohomology. This ensures that the stable cohomology of $\Asigma$ is a graded Hopf algebra and thus, by Hopf's theorem~\cite[Theorem 3C.4]{hatcher-book}, it must be a free graded-commutative algebra. By Theorem~\ref{thm:oldmain}, all generators of stable cohomology have even degree, hence $\coh{\Asigma[\infty]}$ is a polynomial algebra. The goal of the present paper is to understand how to obtain the number of generators of this polynomial algebra algorithmically.

\section{Stable classes and group invariants}\label{sec:compute}
We now develop the machinery of working with representations of wreath products $G\wr S_n$ that is necessary to compute the contributions of the individual strata to the stable cohomology. The stable cohomology of the strata $\stratum\sigma$ is computed in terms of the invariants of the action of the refined automorphism group of the cone. We first investigate how such refined automorphism groups can be computed for direct sums of cones.
\begin{df}
A cone $\sigma$ of an additive collection $\ssigma$ is called \emph{reducible} if it is $\GL(g,\ZZ)$-equivalent to the direct sum $\sigma_1\oplus\sigma_2$ of some non-zero cones $\sigma_1\in\Sigma_{g_1}$ and $\sigma_2\in\Sigma_{g_2}$; otherwise it is called {\em irreducible}.
\end{df}
As usual, one proves by induction that every cone $\sigma\in\ssigma$ can be written as a direct sum of irreducible cones, uniquely up  to reordering the summands. We will adopt the notation $[\sigma]\in[\ssigma]$ for the $\GL(g,\ZZ)$-orbit of a cone, and will write such a decomposition into irreducible cones as
\begin{equation}\label{eq:irrdec}
[\sigma]=[\sigma_1^{\oplus m_1}\oplus\dots\oplus\sigma_\ell^{\oplus m_\ell}]
\end{equation}
where the distinct irreducible summands $[\sigma_1],\dots,[\sigma_\ell]\in[\ssigma]$ and their multiplicities $m_i$ are unique up to reordering. We now compute the refined automorphism group of $\sigma$ in terms of the decomposition into a direct sum of irreducible cones. To formulate the answer, we recall that for a group $G$, for any $n\in\ZZ_{>0}$ the wreath product  $G\wr S_n$ is defined as the semidirect product $G\rtimes S_n$, where the symmetric group $S_n$ acts by permuting the factors.

Then standard computations give the following result:
\begin{lm}
If $[\sigma]=[\sigma_1^{\oplus m_1}\oplus\dots\oplus\sigma_\ell^{\oplus m_\ell}]$ is a decomposition of a cone in $\ssigma$ into a direct sum of irreducible cones of $\ssigma$, then the refined automorphism group of $\sigma$ is
$$
G_\sigma\cong (G_{\sigma_1}\wr S_{m_1}) \times\dots \times (G_{\sigma_\ell}\wr S_{m_\ell}).
$$
\end{lm}
% \begin{proof}
% Denote $r_j:=\rank\sigma_j$ and denote $\sum m_jr_j=g:=\rank\sigma$. If $\alpha\in\GL(g,\ZZ)$ stabilizes $\sigma$, then $\alpha$ can only permute summands of the decomposition that are in the same $\GL(g,\ZZ)$-orbit, hence  $\alpha\in\GL(m_1r_1,\ZZ)\times \dots \times \GL(m_\ell r_\ell,\ZZ)$, and we denote by $\alpha_j\in\GL(m_jr_j,\ZZ)$ the factors of this decomposition. In particular $\alpha_j$ must permute the summands $\sigma_j^{\oplus m_j}$ according to some permutation $p_j\in S_{m_j}$. Let $P_j\in\GL(m_jr_j,\ZZ)$ be the linear operator that permutes the summands of $\RR^{r_j}\oplus\dots\oplus\RR^{r_j}$ according to permutation $p_j$, composed with the identity map on each $\RR^{r_j}$. Then $\alpha_j \circ P_j^{-1}$ is an element of the stabilizer of $\sigma_j^{\oplus m_j}$ which fixes each summand $\sigma_j$, i.e.~is an element of the $m_j$-th direct product of the stabilizer of $\sigma_j$. Thus the stabilizer of $\sigma_j^{\oplus m_j}$ is the wreath product of the stabilizer of $\sigma_j$ with $S_{m_j}$.
% Passing to the refined automorphism groups, we obtain $G_{\sigma_j^{\oplus m_j}}\cong G_{\sigma_j}\wr S_{m_j}$ for each summand $\sigma_j^{\oplus m_j}$ of $\sigma$ and
% $$
% G_\sigma\cong (G_{\sigma_1}\wr S_{m_1}) \times\dots \times (G_{\sigma_\ell}\wr S_{m_\ell})
% $$
% for the refined automorphism group of $\sigma$ itself.
% \end{proof}

This expresses the refined automorphism group of an arbitrary cone in terms of the refined automorphism groups of irreducible cones. The stable cohomology of the corresponding stratum involves computing invariants, which are clearly given as
\begin{equation}\label{eq:symsum}
(\Sym^\pu V_\sigma)^{G_\sigma}
\cong
(\Sym^\pu (V_{\sigma_1}^{\oplus m_1}))^{G_{\sigma_1}\wr S_{m_1}}\otimes \dots \otimes (\Sym^\pu (V_{\sigma_\ell}^{\oplus m_\ell}))^{G_{\sigma_\ell}\wr S_{m_\ell}}.
\end{equation}

We now use these computations to rewrite the stable Borel--Moore homology of $\Asigma[\infty]$, which is the sum over all cones $\sigma\in\ssigma$, as a sum over all irreducible cones only. We express these results in terms of suitable generating series.

For a cone $\sigma\in\Sigma_g$, let  $P_\sigma(t)$ be the generating series for the dimensions of the graded pieces of the stable cohomology of $\stratum\sigma$ as an $R$-module. This is to say, we let
$$
  P_\sigma(t):= \sum_{k=0}^\infty \dim_R\left[\coh{\stratum[\infty]\sigma}\right]_{2k}\; t^k
  = \sum_{k=0}^\infty \dim_\QQ \left(\Sym^\pu V_\sigma\right)^{G_\sigma} t^k,
$$
where $\big[\cdot\big]_\pu$ denotes the graded pieces as an $R$-module.

By Proposition~\ref{prop:degenerate} and Poincar\'e duality for the strata, for any small additive collection $\ssigma$ the stable Borel--Moore homology of $\Asigma$ is simply the sum of the stable cohomology of the individual strata. Thus the generating function of the stable graded pieces as an $R$-module:
$$
  P_\ssigma(t):=\sum_{k=0}^\infty \dim_R\left[\BM[\topd-\pu]{\Asigma[\infty]}\right]_{2k}\; t^k
  = \sum_{[\sigma]\in[\ssigma]}t^{\dim\sigma}P_\sigma(t),
$$
where $\big[\cdot\big]_\pu$ denotes the graded pieces of stable homology of $\Asigma$ as an $R$-module. Then the lemma above and formula~\eqref{eq:symsum} imply
$$
  t^{\dim\sigma}P_{\sigma}(t) = \prod_{j=1}^\ell t^{m_j\dim\sigma_j}P_{\sigma_j^{\oplus m_j}}(t).
$$

We write $\ssigma_{\operatorname{irr}}$ for the collection of irreducible cones in $\ssigma$, and write $[\ssigma]_{\operatorname{irr}}$ for the collection of orbits of irreducible cones. Since every cone $\sigma\in\ssigma$ is a sum of finitely many irreducible cones, we obtain
\begin{equation}\label{eq:infiniteprod}
\begin{aligned}
P_{\ssigma}(t) &= \sum_{\substack{m:\;{[\ssigma]}_{\op{irr}}\rightarrow \NN \\ m(\sigma)=0 \text{ for all but finitely many }[\sigma]}}
\prod_{[\sigma]\in{[\ssigma]}_{\op{irr}}} t^{m(\sigma)\dim\sigma}P_{\sigma^{\oplus m(\sigma)}}(t)\\
&=
\prod_{[\sigma]\in{[\ssigma]}_{\op{irr}}}\left(\sum_{m\geq 0}t^{m\dim\sigma}P_{\sigma^{\oplus m}}(t)\right).
\end{aligned}
\end{equation}
Thus to obtain an algorithm for computing the stable homology of $\Asigma[\infty]$, it remains to understand the factors in the last line. This is really a question in representation theory, and we study it in that generality.
\begin{df}
For a finite group $G$ acting linearly on a finite-dimensional $\QQ$-vector space $V$, we define the formal power series
$$
P_{(G,V)}(t) := \sum_{k=0}^\infty \dim_\QQ (\Sym^kV)^G t^k.
$$
\end{df}
Molien's formula~\cite[Thm.~1.10]{mukai-invariants} allows one to compute this generating series explicitly as
\begin{equation}\label{eq:molien}
P_{(G,V)}(t)=\frac1{\#G } \sum_{A\in G}\frac 1{\det({\bf 1} - tA)}.
\end{equation}
We will now compute $P_{(G\wr S_n,V^{\oplus n})}(t)$ for any $n\in\NN$ in terms of  $P_{(G,V)}(t)$. The wreath product  $G\wr S_n$ acts naturally on $V^{\oplus n}$, whereby $G$ acts on each summand, and $S_n$ permutes the summands.
%Thus the action of any element of $G\wr S_n$ can be written as a matrix in $GL(V^{\oplus n},\QQ)$ of the form $MP$, where $M$ is a matrix in $G^{n}$, embedded in $\GL(V^{\oplus n};\QQ)$, and $P$ is a matrix that permutes the $n$ direct summands $V$, and acts by identity on each summand. We can thus obtain
We then have the following computation (see \cite[Thm 6.4]{dist} for a proof).

\begin{prop}\label{prop:wreath}
 For any finite group $G$ acting on a $\QQ$-vector space $V$ we have the expression
$$
  P_{(G\wr S_n,V^{\oplus n})}(t) = \frac 1{n!} \sum_{\substack{1\leq i_1\dots\leq i_r\\i_1+\dots+i_r=n}} c_{i_1,\dots,i_r}P_{(G,V)}(t^{i_1})\cdot P_{(G,V)}(t^{i_2})\cdot\dots\cdot P_{(G,V)}(t^{i_r})
$$
where $c_{i_1,\dots,i_r}$ is the number of permutations of cycle type $(i_1,\dots,i_r)$ in $S_n$.
\end{prop}

\section{The plethystic substitution and the proof of the main Theorem}\label{sec:proof}
The expression given by  Proposition \ref{prop:wreath} can be stated in terms of the plethystic substitution, the notion of which we now recall (see \cite[\S3.3]{haiman-symmfunct} for more details). Restating it this way will help us state the main theorem in the most compact form, and prove it.
\begin{df}
We denote by
$$
  \Lambda:=\underset{n}{\underleftarrow{\lim}}\;\ZZ[X_1,\dots,X_n]^{S_n}
$$
the ring of symmetric functions in infinitely many variables, and for any $n\in\NN$ call the sum of all degree $n$ monomials
$$
 h_n:=X_1^n+X_1^{n-1}X_2+\dots\in\Lambda
$$
the {\em complete homogeneous polynomial}, and call
$$
 p_n:=X_1^n+X_2^n+\dots\in\Lambda
$$
the {\em power sum}.
\end{df}
It is well-known that the set $\{h_n\}$ freely generates $\Lambda$, while  $\{p_n\}$ freely generates the $\QQ$-algebra $\Lambda\otimes_\ZZ \QQ$. These two sets are related by the formula
$$
  h_n(X)=\frac 1{n!} \sum_{\substack{1\leq i_1\leq\dots\leq i_r\\i_1+\dots+i_r=n}} c_{i_1,\dots,i_r} p_{i_1}(X) p_{i_2}(X)\dots p_{i_r}(X).
$$
\begin{df}
Given a symmetric function $q(X)\in\Lambda$ and a formal power series in one variable with integral coefficients $P(t)\in\ZZ[[t]]$, the plethystic substitution $q[P](t)\in\QQ[[t]]$ is defined as follows. First, since monomials in $p_n$'s form a basis of $\Lambda\otimes_\ZZ\QQ$, one can uniquely write $q$ as their linear combination, denoting the coefficients $\alpha_{i_1,\dots,i_r}$, so that
$$
  q=\sum_{r=1}^{\deg q}\sum_{i_1,\dots,i_r} \alpha_{i_1,\dots,i_r} p_{i_1}\cdot\dots\cdot p_{i_r}.
$$
Then the plethystic substitution is defined as the power series
$$
  q[P](t) := \sum_{r=1}^{\deg q}\sum_{i_1,\dots,i_r} \alpha_{i_1,\dots,i_r} P(t^{i_1}) \cdot \dots \cdot P(t^{i_r}).
$$
\end{df}
Thus Proposition~\ref{prop:wreath} can be restated by saying that $P_{(G\wr S_n,V^{\oplus n})}(t)$ equals the plethystic substitution $h_n[P_{(G,V)}](t)$.

The plethystic substitution can be computed more explicitly as follows.
\begin{lm}\label{lm:plethysm}
Suppose all coefficients of the non-zero power series $0\neq P(t)=c_0+c_1 t +c_2 t^2+\dots$ are non-negative integers $c_i\ge 0$. Then $q[P](t)$ can be obtained from $q(X_1,X_2,\dots,X_i,\dots)$ by substituting  $X_1=\dots = X_{c_0}=1$, $X_{c_0+1 }=\dots = X_{c_0+c_1}=t$, \dots, $X_{c_0+\dots+ c_{j-1}+1}=\dots=X_{c_0+\dots+ c_j}=t^j,\dots$.
\end{lm}
\begin{proof}
Since the plethystic substitution is linear in $q$, and evaluating is also linear in $q$, it is enough to prove the lemma for monomials $p_{i_1}\cdot\dots\cdot p_{i_r}$. For such a monomial the evaluation is the product of evaluations of the $p_{i_j}$, while clearly by definition the plethystic substitution is also the product $(p_{i_1}\cdot\dots\cdot p_{i_r})[P](t)=p_{i_1}[P](t)\cdot\dots\cdot p_{i_r}[P](t)$. Thus it is enough to prove the lemma for the case when $q=p_j$ for some $j$, in which case it is clear.
\end{proof}
\begin{df}
For a formal power series with integral coefficients $P(t)\in\ZZ[[t]]$ we define its exponential to be the plethystic substitution
$$
  (\Exp(P))(t):=\sum_{n\geq 0}h_n[P][t].
$$
\end{df}
Since $h_n$ is equal to the degree $n$ part of the infinite product
$$\prod_{i=1}^\infty (1+X_i^2+\dots+X_i^n),$$
%$$(1+X_1^2+\dots+X_1^n)\cdot(1+X_2^2+\dots+X_2^n)\cdot \dots \cdot (1+X_k^2+\dots+X_k^n)\cdot \dots,$$
we have the following formal relation:
\begin{equation}\label{eq:series-h_n}
\sum_{n=0}^\infty h_n = \prod_{n=1}^\infty \frac1{1-X_n}.
\end{equation}
Thus, if all coefficients of the non-zero power series $0\neq P(t)=c_0+c_1 t +c_2 t^2+\dots$ are non-negative integers, it follows from Lemma~\ref{lm:plethysm} that
$$
\Exp\left(\sum_{i=0}^{\infty}c_it^i\right) = \prod_{i=1}^\infty\frac1{(1-t^i)^{c_i}}.
$$
Hence,  in this terminology Proposition~\ref{prop:wreath} can be finally restated as follows.

\begin{prop}\label{prop:wrexp}
For any finite group $G$ acting on a $\QQ$-vector space $V$, and for any $\alpha\in\ZZ_{>0}$, the generating series for Betti numbers can be computed as follows:
$$
\sum_{n=0}^\infty t^{ n\alpha}P_{(G\wr S_n,V^{\oplus n})}(t)
=\Exp(t^\alpha P_{(G,V)}(t))=
\prod_{k=1}^\infty \frac1{(1-t^k)^{c_k}},
$$
where the coefficients $c_k$ are defined by $t^\alpha P_{(G,V)}(t)=\sum_{k=1}^\infty c_k t^k $.
\end{prop}
The arbitrary parameter $\alpha$ serves to give a version of the formula that is general enough for our purposes.
%intermediate steps
%$$
%\sum_{n=0}^\infty t^{ n\alpha}P_{(G\wr S_n,\bigoplus^nV)}(t) =
%\sum_{n=0}^\infty t^{ n\alpha} h_n[P_{(G,V)}](t)=
%\sum_{n=0}^\infty h_n[t^\alpha P_{(G,V)}](t)=
%\prod_{k=1}^\infty \frac1{(1-t^k)^{c_k}}
%$$
We can now prove our main result.
\begin{proof}[Proof of Theorem~\ref{thm:algorithm}]
Since the statement of the theorem concerns only the structure of $\BM[\topd-\pu]{\Asigma[\infty]}$ as a free graded $R$-module, it suffices to verify that the dimension of $\BM[\topd-k]{\Asigma[\infty]}$ agrees with the dimension of the degree $k$ part of the symmetric algebra of $V_\pu^\ssigma$ for each $k$. Hence it is enough to prove that $P_\ssigma(t)$ coincides with the generating function of the dimension of a polynomial algebra with $\dim_\QQ V_{2k}^\ssigma$ generators in each degree $k$.

Substituting the result of Proposition~\ref{prop:wrexp} in the expression for $P_\ssigma(t)$ given by formula~\eqref{eq:infiniteprod} we obtain
\begin{align*}
P_{\ssigma}(t) &=\prod_{[\sigma]\in{[\ssigma]}_{\op{irr}}}\left(\sum_{m\geq 0}t^{m\dim\sigma}P_{\sigma^{\oplus m}}(t)\right)
=\prod_{[\sigma]\in{[\ssigma]}_{\op{irr}}}\Exp(t^{\dim\sigma}P_{\sigma}(t))\\
&=
\Exp\left(\sum_{[\sigma]\in{[\ssigma]}_{\op{irr}}}t^{\dim\sigma}P_{\sigma}(t)\right)
=\prod_{k=1}^{\infty}\frac1{(1-t^k)^{c_k}},
\end{align*}
where $c_k$ is the coefficient of $t^k$ in $\sum_{[\sigma]\in{[\ssigma]}_{\op{irr}}}t^{\dim\sigma}P_{\sigma}(t)$. By definition of $P_{\sigma}(t)$, the coefficient $c_k$ is equal to the dimension of the direct sum $V_{2k}^\ssigma=\bigoplus_{[\sigma]\in{[\ssigma]}_{\op{irr}}}
(\Sym^{k-\dim\sigma} V_\sigma)^{G_\sigma}$. This concludes the proof.
\end{proof}

\section{Transversality}\label{sec:transverse}
In this section we prove the transversality result for additive collections of fans, in particular  Proposition~\ref{prop:transversal}.
We shall do this in the analytic category, which is sufficient for our purposes.

We have already recalled that, in order to construct a toroidal compactification of $\ab$,
one has to firstly construct partial compactifications in the direction of the cusps, and secondly glue these partial compactifications to obtain a global space. In the case of principal polarization  the cusps correspond to the (up to the action of $\Sp(2g,\ZZ)$ unique) $k$-dimensional isotropic subspaces $U_k \subset \QQ^{2g}$ spanned by the first $k$ elements of the standard basis. We shall denote such a cusp by $F_k$. For any $0<k\le g$ the stabilizer of $U_k$ defines a maximal parabolic subgroup $P_k \subset \Sp(2g,\ZZ)$, which can be described explicitly in terms of generators, see eg.~\cite[\S7]{grhuto}. We denote by $P'_k$ the center of the unipotent radical of $P_k$. The group $P'_k$ is then a lattice of rank $k(k+1)/2$, and taking the partial quotient with respect to $P'_k$ defines a map
\begin{equation}\label{equ:partquot}
  \mathbb H_g \to \mathbb H_{g-k} \times \CC^{k(n-k)} \times (\CC^*)^{k(k+1)/2}.
\end{equation}
Further note that the fan $\Sigma_k$ defines a torus embedding
\begin{equation}\label{equ:torusemb}
  (\CC^*)^{k(k+1)/2} \subset T_{\Sigma_{k}}
\end{equation}
which will eventually provide the partial compactification in the direction of $F_k$.

Dividing by the Jacobi group in the parabolic subgroup $P_k$, see~\cite[\S7]{grhuto}, the partial quotient map (\ref{equ:partquot}) descends to an inclusion
\begin{equation}\label{equ:cuspinc}
\xymatrix@C=66pt
{
\mathcal T_{k,g} \ar[r]\ar[d]%^{q(\sigma)}
%\ar@/_3ex/[dd]_{\pi(\sigma)}
&\mathcal T_{\Sigma_{k,g}}\ar[d]%_{q(\sigma)}
%\ar@/^3ex/[dd]^{\pi(\sigma)}
\\
\ua[g-k]^{\times k}\ar[r]
\ar[d]%^{p_i}
&\ua[g-k]^{\times k}\ar[d]%_{p_i}
\\
{\ab[g-k]}\ar[r]&{\ab[g-k]}.
}
\end{equation}
Here $\ua[g-k] \to \ab[g-k]$ denotes the universal family and the fiber of $\mathcal T_{\Sigma_{k,g}}\ \to \ua[g-k]^{\times k}$ is the
toric variety $T_{\Sigma_k}$. The variety $\calT_{k,g}$ is the image of $\HH_g$ under  the partial quotient map (\ref{equ:partquot}).
For an admissible collection of fans, the toric variety $T_{\Sigma_{k}}$ only depends on $k$, but not on $g$. Strictly speaking, the above construction should be performed over a level cover of $\ab$ such that the arithmetic group is neat --- which from now on we will tacitly assume to be the case. To obtain the partial compactification, we take the interior of the closure of  $\mathcal T_{k,g} $ in $\mathcal T_{\Sigma_{k,g}}$ and denote this by $\overline{\mathcal T}_{k,g}$. Finally, in order to obtain a neighborhood in $\Asigma$ in the direction of the cusp $F_k$, one has to take the quotient of $\overline{\mathcal T}_{k,g}$ by the group $\GL(g-k,\ZZ)$, see again \cite[\S7]{grhuto}. However, since this group acts freely, provided we have a sufficiently big level structure, we can disregard this group action for the purposes of our proof.

We can also describe the new boundary strata which we have added by this process. These are enumerated by cones $\sigma \in \Sigma_k$ which contain rank $k$ matrices. The strata associated to the  cones of smaller rank already occur in the compactification process for the cusps $F_{k'}$ for some $k' <k$. In order to describe these new strata we first notice that the torus $(\CC^*)^{k(k+1)/2}$ which appears in equations (\ref{equ:partquot}) and (\ref{equ:torusemb})   can be identified as
\begin{equation}
  \TT_k=  \Sym^2(\ZZ^{k}) \otimes \CC^* = (\CC^*)^{k(k+1)/2}
\end{equation}
where $\Sym^2(\ZZ^{k}) \subset \Sym^2(\RR^{k})$ and $\Sigma_k$ is a decomposition of $\Sym^2(\RR^{k})$. A cone $\sigma$ which contains rank $k$ matrices defines a torus
$T(\sigma)= \TT_i / \TT_{\sigma}$,
where
$\TT_{\sigma}= (\Sym^2(\ZZ^{k}) \cap \Span(\sigma))\otimes \CC^*$,
%$\Sym^2(\ZZ^{k}) \otimes \CC^*/ (\Sym^2(\ZZ^{k}) \cap \Span(\sigma))\otimes \CC^*$
of dimension $k(k+1)/2 - \dim(\sigma)$.
The stratum $\beta(\sigma)$ associated to $\sigma$ is then the double fibration
%\begin{equation}\label{equ:fibrestratum}
%\xymatrix@C=66pt
%{
%{\beta(\sigma) \ar[d]}%^{q(\sigma)}
%\ar@/_3ex/[dd]_{\pi(\sigma)}
%&{}%_{q(\sigma)}
%\ar@/^3ex/[dd]^{\pi(\sigma)}
%\\
%{}
%{\ua[g-k]^{\times k}\ar[d]}%^{p_i}
%&{}%_{p_i}
%\\
%{\ab[g-k]}&{}.
%}
%\end{equation}
\begin{equation}\label{equ:fibrestratum}
  \beta_g(\sigma) \to \ua[g-k]^{\times k} \to \ab[g-k],
\end{equation}
where the first fibration $\beta_g(\sigma) \to \ua[g-k]^{\times k}$ is a torus bundle with fiber  $T(\sigma)$.

\medskip

\begin{df}
We say that an embedding $i: X_1 \times X_2 \to X$ of analytic varieties is {\em transverse} if the following holds:
for every point $(x_1,x_2) \in X_1 \times X_2$ there exist neighborhoods $U_1$ and $U_2$ of $x_1$ and $x_2$ in $X_1$ and $X_2$ respectively, as well as a ball $B \subset \CC^d$ and an isomorphism $\varphi: U_1 \times U_2 \times B \to W$ to a neighborhood $W$ of $i(x_1,x_2)\in X$, such that
$$
  \varphi^{-1} \circ i|_{U_1 \times U_2}: U_1 \times U_2 \to U_1 \times U_2 \times B, \, (x_1,x_2) \mapsto (x_1,x_2,0).
$$
\end{df}
The map $(\sigma_1,\sigma_2)\mapsto \sigma_1\oplus\sigma_2$ for an additive collection induces an embedding
\begin{equation}\label{equ:embeddingtorus}
  i_{k,g-k}: T_{\Sigma_k} \times T_{\Sigma_{g-k}} \to T_{\Sigma_g}.
\end{equation}
Indeed this is  a transverse embedding: if
\begin{equation}
\TT_{k,g-k}=\Sym^2(\ZZ^{g}) \otimes \CC^*/(\Sym^2(\ZZ^{k}) \oplus \Sym^2(\ZZ^{g-k})) \otimes \CC^*),
\end{equation}
then the embedding $i_{k,g-k}$ is given by the map
\begin{equation}\label{equ:localtransversal}
i_{k,g-k}: T_{\Sigma_k} \times T_{\Sigma_{g-k}} \to T_{\Sigma_k} \times T_{\Sigma_{g-k}} \times \TT_{k,g-k} \subset T_{\Sigma_g}
\end{equation}
that sends $(x_1,x_2)$ to $(x_1,x_2,1)$.

We are now ready to prove the main result of this section.
%\begin{rem}
%One can describe the normal directions of the transverse embedding  ${\mathcal A}_{k} \times {\mathcal A}_{g-k} \to {\mathcal A}_{g}$ very explicitly. For this consider the following
%subdivision of a genus $g$ period matrix
%$$
%\tau_g = \begin{pmatrix}
%\tau_{k} & \tau_{k,g-k}\\
%{}^{t}\tau_{k,g-k} & \tau_{g-k}
%\end{pmatrix}.
%$$
%Then the embedding  ${\mathcal A}_{k} \times {\mathcal A}_{g-k} \to {\mathcal A}_{g}$ is locally given by
%$$
%(\tau_k, \tau_{g-k})  \mapsto
%\begin{pmatrix}
%\tau_{k} & 0\\
%0 & \tau_{g-k}
%\end{pmatrix}
 %$$
 %and the coordinates in the normal direction are given by the entries of the matrix $\tau_{k,g-k}$.
%\end{rem}

\begin{proof}[Proof of Proposition~\ref{prop:transversal}]
By the additivity of the collection of admissible fans and the construction of toroidal compactifications, we clearly have a map
$\Asigma[k] \times\Asigma[g-k]\to\Asigma$. To analyze this map
in more detail, and prove transversality, we consider a neighborhood of a cusp $F_{k'}$ in $\Asigma[k]$ and a neighborhood of a cusp $F_{k''}$ in  $\Asigma[g-k]$. As our claim is local in nature, and since we are on a suitable level cover, it is enough to understand  the diagram
\begin{equation}%\label{diagram-beta}
\xymatrix@C=66pt
{
\overline{\mathcal T}_{k',k} \times \overline{\mathcal T}_{k'',g-k} \subset {\mathcal T}_{\Sigma_{k',k}}  \times {\mathcal T}_{\Sigma_{k'',g-k}} \ar[r]\ar[d]%^{q(\sigma)}
%\ar@/_3ex/[dd]_{\pi(\sigma)}
&{\overline{\mathcal T}_{k'+k'',g}} \subset {\mathcal T}_{\Sigma_{k' + k'',g}}  \ar[d]%_{q(\sigma)}
%\ar@/^3ex/[dd]^{\pi(\sigma)}
\\
\ua[k-k']^{\times k'} \times \ua[g-k-k'']^{\times k''}\ar[r]
\ar[d]%^{p_i}
&{\ua[g-k'-k'']^{\times (k'+k'')}}\ar[d]%_{p_i}
\\
{\ab[k-k'] \times \ab[g-k-k'']}\ar[r]&{\ab[g-k'-k'']}.
}
\end{equation}
We  note that this induces maps of strata
\begin{equation*}
\beta(\sigma_1) \times \beta(\sigma_2) \to \beta(\sigma_1 + \sigma_2).
\end{equation*}
Locally the top horizontal arrow is simply the embedding $i_{k,g-k}: T_{\Sigma_k} \times T_{\Sigma_{g-k}} \to T_{\Sigma_g}$ from (\ref{equ:embeddingtorus}),
in particular it is fiberwise transverse by (\ref{equ:localtransversal}). Since all other horizontal maps are also transverse, the claim follows.
\end{proof}

\section{Examples of applications of the algorithm}\label{sec:examples}
In this section we apply the algorithm described above to explicitly compute the stable homology of various partial compactifications.

The simplest additive collections are those generated by just one cone and its direct sums with itself. The simplest cone is $\sigma_1$, the corresponding stratum for which is the boundary of Mumford's partial toroidal compactification, which is a subset of any toroidal compactification. The cones that are direct sums of the form $\sigma_{1+\dots+1}$ correspond to the standard degenerations of abelian varieties and were called {\em standard cones} in~\cite{grhuto}; the resulting partial compactification $\Std$  was called there the {\em standard partial compactification}.

We note that the partial compactification $\Std$ is rationally smooth, and hence by Poincar\'e duality, cap product with the fundamental class defines an isomorphism between cohomology and Borel--Moore homology  $H^k(\Std) \cong \bar H_{\topd -k}(\Std)$  in complementary dimensions. We shall make use of this here and, similarly, for other collections of simplicial cones. In~\cite{grhuto} we computed the stable cohomology of $\Std$, which required significant work. Now this is straightforward.

To apply our machinery, we note that the refined automorphism group $G_{\sigma_1}$ is trivial, and thus $(\Sym^\pu V_{\sigma_1})^{G_{\sigma_1}}$ is simply the polynomial ring $\QQ[T]$, and its Poincar\'e series is thus $P_{\sigma_1}(t)=(1-t)^{-1}$. Thus the generating series for the number of generators of stable cohomology of $\Std$, as an $R$-algebra, is  $1+(1-t)^{-1} t= 1+t+t^2+\dots$. This is to say that the stable cohomology of the partial compactification by standard cones is freely generated, as an $R$-module, by a collection of generators, one in each even degree. In~\cite{grhuto} we identified these generators with the fundamental classes of the strata of $\Std$; however, other choices of generators are possible.

Next, one naturally looks at additive collections that are obtained as direct sums of two cones. It is then natural to consider the only two irreducible cones of rank up to 2, that is $\sigma_1$ and $\sigma_{K_3}$. The graph $K_3$ is the complete graph on three vertices, which can also be thought of as the cyclic graph $C_3$ on 3 vertices. More generally, for any $k\ge 2$ the refined automorphism group $G_{\sigma_{C_k}}$ is the full permutation group $S_k$, which  permutes the rays of the cone. Thus $V_{\sigma_{C_k}}^{G_{\sigma_k}}$ is simply the ring of symmetric functions in $k$ variables. This ring is generated by elementary symmetric functions, which have degrees $1,\dots,k$, so that the corresponding Poincar\'e series is
\begin{equation}\label{PoincareCk}
P_{\sigma_{C_k}}(t)=\prod_{i=1}^k\frac1{1-t^k}.
\end{equation}
Thus the generating function for the number of generators of the stable cohomology of the partial compactification by the cones of the form $\sigma_1^{\oplus a}\oplus\sigma_{K_3}^{\oplus b}$, as an $R$-algebra, is
\begin{multline*}
P_{\{\sigma_1^{\oplus a}\oplus\sigma_{K_3}^{\oplus b}\}}(t)=1+\frac{t}{1-t}+\frac{t^3}{(1-t)(1-t^2)(1-t^3)}=\frac{1-t^2+t^5}{(1-t)(1-t^2)(1-t^3)}
\\
=
1 + t + t^{2}  + 2 t^{3}  + 2 t^{4}  + 3 t^{5}  + 4 t^{6}  + 5 t^{7}  + 6 t^{8}  + 8 t^{9}  + 9 t^{10}   + 11 t^{11}
+13 t^{12}   + 15 t^{13} +\cdots%  + 17 t^{14}   + 20 t^{15}   + 22 t^{16}   + 25 t^{17}   + 28 t^{18}   + 31 t^{19}   + 34    t^{20}   + \dots.
\end{multline*}

\medskip
We now proceed to deal with the matroidal partial compactification. Recall that the cones of the matroidal fan correspond to simple regular matroids. All such matroids occurring in our range of degrees are graphical; we label them by the corresponding graph, and list the 1-sum-irreducible ones in Table~\ref{t:matrcones7}.

\begin{table}[ht]
%\scalebox{0.96}{
  $
\begin{array}{clcc}
\text{cone}&\text{generators}&\text{dim.} & \text{rank}\\
\hline
\\[-2ex]
\sigma_1 & x_1^2
&1&1
\\
\sigma_{K_3} &  x_1^2,x_2^2,(x_1-x_2)^2
&3&2
\\
\sigma_{C_4} & x_1^2,x_2^2,(x_1-x_3)^2,(x_2-x_3)^2
&4&3
\\
\sigma_{K_4-1} &  x_1^2,x_2^2,x_3^2,(x_1-x_3)^2,(x_2-x_3)^2
&5&3
\\
\sigma_{C_5} & x_1^2, x_2^2, (x_1-x_4)^2,(x_2-x_3)^2,(x_3-x_4)^2
&5&4
\\
\sigma_{K_4} & x_1^2,x_2^2,x_3^2,(x_1-x_2)^2,(x_1-x_3)^2,(x_2-x_3)^2
&6&3
\\
\sigma_{C_{222}} &  x_1^2, x_2^2, x_3^2, (x_1-x_4)^2,(x_2-x_4)^2,(x_3-x_4)^2
&6&4
\\
\sigma_{C_{321}} &  x_1^2, x_2^2, x_4^2, (x_1-x_4)^2,(x_2-x_3)^2,(x_3-x_4)^2
&6&4
\\
\sigma_{C_6} &  x_1^2, x_2^2, (x_1-x_5)^2,(x_2-x_3)^2,(x_3-x_4)^2,(x_4-x_5)^2
&6&5\\
\sigma_{C_{2221}} & x_1^2, x_2^2, x_3^2, x_4^2, (x_1-x_4)^2,(x_2-x_4)^2,(x_3-x_4)^2
&7&4\\
\sigma_{K_5-2-1} & x_1^2, x_2^2, x_4^2, (x_1-x_2)^2, (x_1-x_4)^2,(x_2-x_3)^2,(x_3-x_4)^2
&7&4\\
\sigma_{K_5-3} &  x_1^2, x_2^2, x_3^2, (x_1-x_3)^2, (x_1-x_4)^2,(x_2-x_3)^2,(x_3-x_4)^2
&7&4\\
\sigma_{C_{421}} & x_1^2, x_2^2, x_3^2, (x_1-x_5)^2,(x_2-x_3)^2,(x_3-x_4)^2,(x_4-x_5)^2
&7&5\\
\sigma_{C_{331}} &  x_1^2, x_2^2, x_4^2, (x_1-x_5)^2,(x_2-x_3)^2,(x_3-x_4)^2,(x_4-x_5)^2
&7&5\\
\sigma_{C_{322}} & x_1^2,x_2^2,x_5^2,(x_1-x_3)^2,(x_2-x_3)^2,(x_3-x_4)^2,(x_4-x_5)^2
&7&5\\
\sigma_{C_7} & \begin{array}{@{}l} x_1^2, x_2^2, (x_1-x_6)^2,(x_2-x_3)^2,(x_3-x_4)^2,(x_4-x_5)^2, \\(x_5-x_6)^2\end{array}
&7&6\\
\end{array}
$
%}
\caption{\label{t:matrcones7}
Representatives of the equivalence classes of irreducible matroidal cones of dimension up to $7$.}
\end{table}

To apply our machinery, we need to know the refined automorphism group of each of these cones. Formula~\eqref{PoincareCk} gives the Poincar\'e series for the cyclic cones $C_3=K_3$, $C_4,C_5,C_6$. For the other cones, one can use the fact that the refined automorphism group of a matroidal cone is isomorphic to the automorphism group of the corresponding matroid.
An explicit description of the refined automorphism group for cones of dimension at most $6$ can also be extracted from~\cite{grhuto}.
An important feature is that all cones we are working with are basic cones. For this reason, the $G_\sigma$ always acts as a permutation representation on $V_\sigma$. The computations are straightforward but tedious; we thus give one sample detailed computation, and summarize the results for the other cones in a table.

{\em For the cone $\sigma_{K_4-1}$}, it follows from~\cite[\S6.5]{huto2} that $G_{\sigma_{K_4-1}}$ is the stabilizer of $x_3^2$ in the refined automorphism group of $\sigma_{C_4}$. Thus the group $G_{\sigma_{K_4-1}}$ is generated by the three involutions
$$
(x_1^2\leftrightarrow (x_1-x_3)^2), \ (x_2^2\leftrightarrow (x_2-x_3)^2)
\text{ and }
(x_1^2\leftrightarrow x_2^2,(x_1-x_3)^2\leftrightarrow (x_2-x_3)^2)$$
and can be identified with a subgroup of $S_4=G_{\sigma_{C_4}}$ consisting of the identity, two transpositions, three permutations of type $(2,2)$ and two $4$-cycles.

Molien's formula \eqref{eq:molien} then yields
$$
P_{K_4-1}(t)=
\frac {1-t+t^2} {(1-t)^3(1-t^2)(1-t^4)}
=\frac {1-t^6} {(1-t)^2(1-t^2)^2(1-t^3)(1-t^4)}
$$
which is to say that the ring of invariants $(\Sym^\pu V_{\sigma_{K_4-1}})^{G_{\sigma_{K_4-1}}}$ has two generators in degree one, two in degree 2, and one each in degrees $3$ and $4$, with one relation in degree 6.
Alternatively, one can describe $G_{\sigma_{K_4-1}}$ as the wreath product $S_2\wr S_2$ and use Proposition~\ref{prop:wreath} to obtain
$$\psigma{K_4-1}=\frac 1{1-t}h_2\left[\frac 1{(1-t)(1-t^2}\right]=\frac {1-t^6} {(1-t)^2(1-t^2)^2(1-t^3)(1-t^4)}.$$

The results for the other cones are given in Table~\ref{table:Gsigma}. In the notation of the generators of the groups in that table, the index $i$ refers to the generator of the cone $\sigma$ appearing in position $i$ in the description given in Table~\ref{t:matrcones7}.

\begin{table}[ht]
$
\begin{array}{ccc}
\text{cone }\sigma& G_\sigma\subset S_{\dim \sigma}&P_\sigma(t)\\
\hline
\\[-2ex]
\sigma_{K_4}&S_4\cong\langle (24)(35),(123)(456)\rangle \subset S_6&\frac{1+t^3+t^4+t^5+t^6+t^9}{(1-t)(1-t^2)^2(1-t^3)^2(1-t^4)}\\[1ex]
\sigma_{C_{222}}&S_2\wr S_3\subset S_6&\frac{1+t^4+t^5-t^7-t^8-t^{12}}{(1-t)(1-t^2)^2(1-t^3)^2(1-t^4)(1-t^6)}\\[1ex]
\sigma_{C_{2221}}&(S_2\wr S_3)\times S_1\subset S_7& \frac{P_{\sigma_{C_{222}}}(t)}{1-t}\\[1ex]
\sigma_{C_{321}}&S_3\times S_2\times S_1\subset S_6&\frac1{(1-t)^3(1-t^2)^2(1-t^3)}\\[1ex]
\sigma_{K_5-3}&D_8\cong\langle (26),(2567)(34)\rangle\subset S_7&\frac{1-t+2t^2}{(1-t)^4(1-t^2)^2(1-t^4)}\\[1ex]
\sigma_{K_5-2-1}&V_4\times S_2\cong\langle (23)(45),(24)(35),(67)\rangle \subset S_7&\frac{1-t+t^2}{(1-t)^4(1-t^2)^3}\\[1ex]
\sigma_{C_{421}}&S_4\times S_2\times S_1\subset S_7&\frac 1{(1-t)^3(1-t^2)^2(1-t^3)(1-t^4)}\\[1ex]
%\sigma_{C_{331}}&\text{index 2}\supset S_3\times S_3&\frac{1-t+t^2+t^4}{(1-t)^3(1-t^2)(1-t^3)(1-t^4)(1-t^6)}\\[1ex]
\sigma_{C_{331}}& (S_3\wr S_2)\times S_1\subset S_7&\frac{1-t+t^2+t^4}{(1-t)^3(1-t^2)(1-t^3)(1-t^4)(1-t^6)}\\[1ex]
\sigma_{C_{322}}&S_3\times G_{\sigma_{K_4-1}}\subset S_7&(1-t)P_{\sigma_{C_3}}(t)P_{\sigma_{K_4-1}}(t)\\[1ex]
\end{array}
$
%}
\caption{\label{table:Gsigma}
Refined automorphism groups and Molien series for irreducible matroidal cones.
}
\end{table}

Combining all of the above data allows us to apply Theorem~\ref{thm:algorithm} to compute the stable cohomology $H^\pu(\Matr[\infty])$ of the matroidal locus in degree up to 16, proving Theorem~\ref{thm:bettimatr} in that range. The results in higher degree are due to Mathieu Dutour Sikiri\'c, who obtained them using the computer algorithm described in the appendix. As noted in the introduction, there is an unfortunate typo in~\cite[Table 1]{grhuto}, which is the source of the discrepancy of the number we obtain now with the numbers in~\cite{grhuto} in degree 12.
\begin{proof}[Proof of Theorem~\ref{thm:bettimatr} for degree up to 16]
Adding up the contributions of irreducible matroidal cones of dimension up to $7$ computed above, we obtain the generating series for the number of generators of stable homology of this partial compactification as an $R$-algebra:
\begin{align*}
  Q(t)&=\!\!\!\sum_{[\sigma]\in[\ssigma^{\op{Matr}}]_{\op{irr}};\ \dim\sigma\leq 7}\!\!\!t^{\dim\sigma}P_\sigma(t) =
  (1- t^{2} + 2t^5+ 5t^{6} + 11t^{7} + 18t^{8} + 35t^{9}\\
  &+ 50t^{10} + 77t^{11} + 102t^{12} + 131t^{13} + 151t^{14} + 173t^{15} + 171t^{16} \\
  &+ 171t^{17} + 151t^{18} + 129t^{19} + 97t^{20} + 74t^{21} + 44t^{22} + 30t^{23}\\
  &+ 14t^{24} + 7t^{25} + 2t^{26} + 2t^{27} -t^{28})
      \cdot\prod_{i=1}^7(1-t^i)^{-1}
      \\
      &=
      1 + t + t^{2} + 2t^{3} + 3t^{4} + 6t^{5} + 13t^{6} + 28t^{7} + 55t^{8} + 113t^{9} + 210t^{10}\\
      &+ 384t^{11} + 663t^{12} + 1109t^{13} + 1776t^{14} + 2778t^{15} + 4196t^{16} + 6209t^{17}\\
      &+ 8958t^{18} + 12691t^{19} + 17621t^{20}+\cdots
\end{align*}
The number of generators of $R$ gives the number of generators of the ring $H^{2k}(\Matr[\infty])$ in each degree $2k\leq 14$.
Moreover, we know that in degree $16$ we have exactly $15$ additional generators, coming from the irreducible $8$-dimensional matroidal cones --- there are two such of rank $4$, eight are of rank $5$, four are of rank $6$ and one is of rank $7$.
This then yields the stable Betti numbers $\dim H^\pu(\Matr[\infty])$ as given in Theorem~\ref{thm:bettimatr}, for degree up to 16.
\end{proof}

\medskip
Finally, we consider the perfect cone compactification, proving Theorem~\ref{thm:bettiperf} for codegree up to 16, by a direct computation. The further computations for codegree up to 22 are due to Mathieu Dutour Sikiri\'c, who obtained them using the computer algorithm described in the appendix. The computer implementation also confirms our manual computations for codegree up to 16 below.
\begin{proof}[Proof of Theorem~\ref{thm:bettiperf} for codegree up to 16]
Up to rank $4$, all cones in the perfect cone decomposition of dimension $\leq 6$ are matroidal. Therefore, they are direct sums of the cones already considered. The number of cones in the perfect cone decomposition of dimension between $5$ and $7$ can be found in
\cite[Fig.~1 \& 2]{numberofperfectforms}.  If one subtracts from these the number of reducible cones and the number of matroidal cones, one obtains that up to isomorphism there exists exactly one irreducible cone of dimension $5$ and rank $5$. In dimension $6$ there exist $2$ orbits of irreducible cones, one of rank $5$ and one of rank $6$. In dimension $7$ there exist $9$ orbits of irreducible cones: two of rank $5$, four of rank $6$ and $3$ if rank $7$.  The explicit form of representatives for these orbits, together with a list of generators for their automorphism groups, was provided to us by Mathieu Dutour~Sikiri\'c (see also \cite{onlinelistcones}).  We list the representatives of these cones in Table~\ref{t:perflist}, where the labeling of the cones is such that the first index indicates the rank, and the second index indicates the dimension.

\begin{table}[ht]
\scalebox{0.90}{
  $\begin{array}{cl}
\text{cone}&\text{generators}\\
\hline
\\[-2ex]
\sigma_{(5,5)} & x_1^2,x_2^2,(x_3+x_4)^2,(x_3+x_5)^2,(x_1+x_2+x_4+x_5)^2
\\
\sigma_{(5,6)} & x_1^2,x_2^2, x_3^2, (x_3+x_4)^2,(x_3+x_5)^2,(x_1+x_2+x_4+x_5)^2
\\
\sigma_{(5,7a)}& x_1^2,x_2^2, x_3^2, x_4^2, (x_3+x_4)^2,(x_3+x_5)^2,(x_1+x_2+x_4+x_5)^2
\\
\sigma_{(5,7b)}& %x_1^2,x_2^2,x_3^2,x_4^2,x_5^2,(x_2-x_3+x_5)^2,(x_1+x_2+x_4-x_5)^2
(x_2-x_3+x_5)^2,x_1^2,x_5^2,x_2^2,x_3^2,x_4^2,(x_1+x_2+x_4-x_5)^2
%
%x_1^2,x_2^2, x_3^2, (x_1+x_4)^2, (x_3+x_4)^2,(x_3+x_5)^2,(x_1+x_2+x_4+x_5)^2
\\
\sigma_{(6,6)} & x_1^2,x_5^2,x_6^2,(x_3+x_4)^2,(x_2+x_3+x_6)^2,(x_1+x_2+x_4+x_5)^2
\\
\sigma_{(6,7a)}& %x_1^2,x_4^2,x_5^2,x_6^2,(x_2+x_3)^2,(x_5+x_6)^2,(x_1-x_2+x_3+x_4+x_5)^2
x_1^2,x_5^2,(x_2+x_3)^2,x_4^2,(x_1-x_2+x_3+x_4+x_5)^2,x_6^2,(x_5+x_6)^2
%
%x_1^2,(x_2+x_3)^2,(x_3+x_4)^2,x_5^2,(x_1+x_2+x_4+x_5)^2,x_6^2,(x_2+x_3+x_6)^2
\\
\sigma_{(6,7b)}& x_1^2,x_2^2,x_5^2,x_6^2,(x_3+x_4)^2,(x_2+x_3+x_6)^2,(x_1+x_2+x_4+x_5)^2
\\
\sigma_{(6,7c)}& x_1^2,x_3^2, x_4^2, (x_4+x_5)^2,(x_5+x_6)^2,(x_1+x_2+x_5)^2,(x_2+x_3+x_6)^2
\\
\sigma_{(6,7d)}& %x_1^2,x_4^2,x_5^2,(x_2+x_3)^2,(x_4+x_6)^2,(x_5+x_6)^2,(x_1-x_2+x_3+x_4+x_5)^2
x_1^2,x_4^2,(x_2+x_3)^2,x_5^2,(x_4+x_6)^2,(x_5+x_6)^2,(x_1-x_2+x_3+x_4+x_5)^2
%
%x_1^2,(x_2+x_3)^2,(x_3+x_4)^2,x_5^2,(x_1+x_2+x_4+x_5)^2,(x_2+x_3+x_6)^2,(x_5+x_6)^2
\\
\sigma_{(7,7a)}&
%& x_6^2,x_7^2,(x_2-x_3)^2,(x_4-x_5)^2,(x_1+x_2-x_6)^2,(x_1-x_3-x_7)^2,(x_3+x_4+x_5)^2
(x_3+x_4+x_5)^2,(x_2-x_3)^2,(x_1-x_3-x_7)^2,(x_1+x_2-x_6)^2,x_6^2,(x_4-x_5)^2,x_7^2
%
%(x_1+x_2+x_3+x_4+x_5+x_6-3x_7)^2, (x_1+x_2-x_7)^2,(x_1+x_3-x_7)^2,\\&(x_2+x_3-x_7)^2,(x_4-x_7)^2,(x_5-x_6)^2,x_4^2
\\
\sigma_{(7,7b)}&
%& x_2^2,x_5^2,x_7^2,(x_3-x_6)^2,(x_4-x_6)^2,(x_1-x_5-x_6)^2,(x_1+x_2+x_3+x_4-x_7)^2
x_5^2,x_7^2,(x_4-x_6)^2,(x_3-x_6)^2,x_2^2,(x_1-x_5-x_6)^2,(x_1+x_2+x_3+x_4-x_7)^2
%
%(x_1+x_2+x_3+x_4+x_5-2x_6-3x_7)^2,(x_1+x_2+x_3+x_4+x_5-2x_6-2x_7)^2,\\
%&(x_1+x_2+x_3+x_4-x_6-2x_7)^2,(x_1+x_2+x_3+x_5-x_6-2x_7)^2,\\
%&(x_1+x_2+x_4+x_5-x_6-2x_7)^2,(x_1+x_3+x_4+x_5-x_6-2x_7)^2,\\&(x_2+x_3+x_4+x_5-x_6-2x_7)^2
\\
\sigma_{(7,7c)}& %x_2^2,x_4^2,x_6^2,x_7^2,(x_1-x_5)^2,(x_3-x_5)^2,(x_1+x_2+x_3+x_4+x_6+x_7)^2
(x_1+x_2+x_3+x_4+x_6+x_7)^2,(x_1-x_5)^2,x_2^2,(x_3-x_5)^2,x_4^2,x_6^2,x_7^2
%
%(x_1+x_2+x_3+x_4+x_5+x_6-3x_7)^2, (x_1-x_7)^2, (x_2-x_7)^2, (x_3-x_7)^2, \\
%&(x_4-x_7)^2, (x_5-x_7)^2, x_6^2
\\
\end{array}$
}
  \caption{\label{t:perflist}
    Representatives of irreducible perfect non-matroidal cones of dimension up to $7$}
\end{table}

It turns out that in this range, the refined automorphism group of the cones is very easy to describe, since $G_\sigma$ acts by permuting the extremal rays of the cones. In all cases but one, this action is simply the direct product of symmetric groups acting on subsets of the set of extremal rays in the natural way.
Therefore, the Molien series is the product of Molien series for elementary symmetric functions.
We list the refined automorphism group in Table~\ref{t:listaut}, where
we denote by  $S_{\lbrace i_1,i_2,\dots,i_r\rbrace}$ the symmetric group permuting the generators of the cone $\sigma$ appearing in position $i_1,\dots,i_r$ in the description given in Table~\ref{t:perflist}. The subgroup $G_{\sigma_{K_4-1}}\subset S_{\{2,4,6,7\}}$ in $G_{\sigma_{(6,7d)}}$ is the subgroup generated by $(24)$, $(67)$ and $(27)(46)$.

\begin{table}[ht]
  $\begin{array}{l|l}
    \sigma & G_{\sigma}\\\hline
    \sigma_{(5,5)} & S_5\\
    \sigma_{(5,6)} & S_{\{1,2,4,5,6\}}\times S_{\{3\}}\\
    \sigma_{(5,7a)} & S_{\{1,2,6,7\}}\times S_{\{3,4\}}\times S_{\{5\}}\\
                \sigma_{(5,7b)} & S_{\{2,6,7\}}\times S_{\{3,4\}}\times S_{\{1,5\}}\\
                \sigma_{(6,6)} & S_6\\
     \sigma_{(6,7a)} & S_{\{1,3,4,5\}}\times S_{\{6,7\}}\times S_{\{2\}}\\
  \end{array}
  \ \
  \begin{array}{l|l}
    \sigma & G_{\sigma}\\\hline
     \sigma_{(6,7b)} & S_{\{1,3,4,5,6,7\}}\times S_{\{2\}}\\
        \sigma_{(6,7c)} & S_{\{1,2,4,6,7\}}\times S_{\{3,5\}}\\
        \sigma_{(6,7d)} & G_{\sigma_{K_4-1}}\times S_{\{1,3,5\}}\\
     \sigma_{(7,7a)} & S_{\{2,3,4,5,7\}}\times S_{\{1,6\}}\\
             \sigma_{(7,7b)} & S_7\\
                \sigma_{(7,7c)} & S_7\\
  \end{array}$
  \caption{\label{t:listaut} Description of the refined automorphism groups of non-matroidal cones of dimension up to $7$.
    }
  \end{table}
Let us define the admissible collection $\ssigma$ by
$$
\sigma \in \ssigma \Leftrightarrow \sigma \text{ is the direct sum of perfect cone cones of dimension }\leq 7.
$$

Then Theorem~\ref{thm:algorithm} applied to the small additive collection $\ssigma$ of simplicial cones gives that $\coh{\Asigma[\infty]}$ is a free $R$-algebra generated by some algebraic classes. Let us denote by $c_k$ the number of generators in degree $2k$ for $k\geq 1$.
The generating function for the $c_k$ is then given by
\begin{align*}
  \sum_{k\geq 0}c_kt^k &= Q(t) + \sum_{\sigma\in \text{table~\ref{t:listaut}}}t^{\dim\sigma}P_\sigma(t)
  %t^5P_{\sigma_{(5,5)}}(t)  + t^6P_{\sigma_{(5,6)}}(t)  + t^6P_{\sigma_{(6,6)}}(t)+t^7\psigma{(5,7a)}
%\\&
%+t^7\psigma{(5,7b)}+t^7\psigma{(6,7a)}+t^7\psigma{(6,7b)}+t^7\psigma{(6,7c)}
%\\&
%+t^7\psigma{(6,7d)}+t^7\psigma{(7,7a)}+t^7\psigma{(7,7b)}+t^7\psigma{(7,7c)}
\\
&
=
(1 - t^{2} + 3t^{5} + 7t^{6} + 21t^{7} + 29t^{8} + 57t^{9} + 80t^{10} + 122t^{11} + 155t^{12}
\\
&
 + 195t^{13} + 215t^{14} + 241t^{15} + 229t^{16} + 223t^{17} + 188t^{18} + 157t^{19} + 113t^{20}
\\
&
 + 84t^{21} + 47t^{22} + 32t^{23} + 14t^{24} + 7t^{25} + 2t^{26} + 2t^{27} - t^{28})
\cdot\prod_{i=1}^7(1-t^i)^{-1}
\\
&
=
1 + t + t^{2} + 2t^{3} + 3t^{4} + 7t^{5} + 16t^{6} + 42t^{7} + 83t^{8} + 177t^{9} + 331t^{10}
\\
&
 + 611t^{11} + 1049t^{12} + 1754t^{13} + 2790t^{14} + 4343t^{15} + 6518t^{16}
\\
&
 + 9596t^{17} + 13759t^{18} + 19400t^{19} + 26792t^{20} +\cdots
\end{align*}
where the sum in the first line goes over all the cones $\sigma$ listed in Table~\ref{t:listaut}.
In particular, if one takes the $\lambda$-classes into account, one obtains
the generating function
$\Exp\left(t/(1-t^2)+\sum_{k\geq 0}c_kt^k\right)$
for the stable Betti numbers of $\Asigma$.

As $\ssigma$ is contained in the perfect cone decomposition, in view of Theorem~\ref{thm:algorithm} the stable Betti numbers of $\Asigma$ give a lower bound for those of $\Perf$.
Moreover, since all cones in the perfect cone decomposition of dimension $\leq 7$ belong to $\ssigma$, the stable Betti numbers of $\Asigma$ and those of $\Perf$ agree in codegree $k\leq 14$. From this we can obtain the Betti number of $\Perf$ in codegree $16$ by recalling from \cite{numberofperfectforms} and \cite{MinkowskianLower8} that there are exactly $53$ irreducible cones of dimension $8$ and rank $\leq 7$ in the perfect cone decomposition: $2$ of rank $4$, $11$ of rank $5$, $16$ of rank $6$, $17$ of rank $7$ and $7$ of rank $8$.
\end{proof}

\section*{Acknowledgements}
We would like to thank Mathieu Dutour Sikiri\'c for his help with the classification of cones of the perfect cone decomposition.

\appendix
\vfil\eject

\section*{Appendix. Computations\\by Mathieu Dutour Sikiri\'c} \label{secappMDS}
The computations of this paper depend on the enumerations of irreducible cones of
fixed dimension $d$ and rank $r$. For the matroidal locus the irreducible cones correspond
to the connected simple loopless regular matroids and those are enumerated
in~\cite{IrreducibleRegularMatroids} up to dimension $15$ and the data is available
from~\cite{onlinelistconesMatroid}.

The enumeration of cones of the perfect cone tessellation is much harder since
there is no translation to pure combinatorics. Table~\ref{NrOrbitCones_Perfect}
gives the number of orbits of cones up to dimension $11$, where $g$ denotes what is called rank in the paper.

Enumeration of the case $d=r$ for $d \leq 8$, respectively $d=9$, is done
in~\cite{MinkowskianLower8}, respectively \cite{Minkowskian9}.
Enumeration of cases $d\leq 6$, respectively $d=7$, is done in~\cite{PerfectFormKtheory56},
respectively~\cite{numberofperfectforms}.
The cases $(r,d)$ = $(8,9)$, $(8,10)$, $(9,10)$ are treated in~\cite{DutourHulekSchuermann}.
The cases $(r,d)$ = $(8,11)$, $(9,11)$, $(10,10)$ and $(10,11)$ have been treated
by the author of this Appendix and the methodology will be published separately.
For the case $(r,d)=(11,11)$ we have only a conjectural list of orbits,
which may be incomplete though this is unlikely.
The complete data sets are available on~\cite{onlinelistcones}.

\begin{table}[!ht]
$
\begin{array}{|c||c|c|c|c|c|c|c|c|}
\hline
&&&&&&&&\\[-2ex]
r\ \backslash d
          & 4 & 5 & 6 & 7  & 8  & 9   & 10   &11\\
          &&&&&&&&\\[-2ex]
\hline
\hline
4         & 1 & 3 & 4 & 4  & 2  & 2   & 2    &-\\
\hline
5         &   & 2 & 5 & 10 & 16 & 23  & 25   &23\\
\hline
6         &   &   & 3 & 10 & 28 & 71  & 162  &329\\
\hline
7         &   &   &   & 6  & 28 & 115 & 467  &1882\\
\hline
8         &   &   &   &    & 13 & 106 & 783  &6167\\
\hline
9         &   &   &   &    &    & 44  & 759  &13437\\
\hline
10        &   &   &   &    &    &     & 283  &16062\\
\hline
11        &   &   &   &    &    &     &      &6674\\
\hline
\end{array}
$
\caption{Number of orbits of cones in the perfect cone decomposition for rank $r\leq 11$ and dimension at most $11$}
\label{NrOrbitCones_Perfect}
\end{table}

We now turn to the algorithms for computing $P_{\sigma}(t)$ for a cone $\sigma$.
The stabilizer of $\sigma$ in $\GL(r,\ZZ)$ is obtained by applying the algorithms
of~\cite{MethodComputGroups}. If $\sigma$ is a face of the matroidal locus then~\cite[Theorem 2]{MethodComputGroups} gives the stabilizers since the unimodularity
of the configuration of vectors of matroidal cones implies that they span $\ZZ^{r}$
and so their stabilizer in $\GL(r,\QQ)$ coincides with their stabilizer in $\GL(r,\ZZ)$.
For a face of the perfect cone tessellation we use the method ``Adding elements to
${\mathcal C}$'' of~\cite[Section 3]{MethodComputGroups} in order to get the group.
That is if $\sigma= \{v_1, \dots, v_m\}$ then we take the matrix $A_{\sigma} = \sum_i v_i v_i^T$.
Then we consider a $\ZZ^{r}$ spanning set of short vectors of $A_{\sigma}^{-1}$. For
example we can take all the vectors lower than the maximum coefficient of $A_{\sigma}^{-1}$.

Once we know the stabilizer, the group $G_{\sigma}$ can be readily computed. But
the group may be large and thus the sum expressing $P_{\sigma}$ impossible to compute
directly. Fortunately the term $\det({\bf 1} - tA)$ is invariant under conjugation.
Therefore we use the enumeration of conjugacy classes of elements of $G_{\sigma}$ in order
to reduce the sum to a more manageable expression. The group computations are done in~\cite{gap}
and the computation with fractions are done in~\cite{pari}.

Finally, we note that the codegree 22 computation on $\Perf$ is so far only conjectural, as the
classification of configurations of vectors of dimension 11 and rank 11 is not finished. In fact
a search for such classifications proceeds by increasing the determinant. The highest known
determinant for such a configuration is 32, and it is known, by a computer search, that there
are no further configurations up to determinant 39 --- but it is only conjectural that the known
list of such configurations of vectors is complete.

%\bibliographystyle{alpha}
%\bibliography{biblio-abelian}
\newcommand{\etalchar}[1]{$^{#1}$}

\end{document}